\documentclass[reqno]{amsart}
\usepackage{amssymb,latexsym,amsmath,amscd,graphicx,graphics,epic,eepic,bm,color,array,mathrsfs,fullpage}
\usepackage{enumerate}
\usepackage{hyperref}
\usepackage[all,knot,poly]{xy}
\usepackage{caption}
\theoremstyle{plain}
\newtheorem{theorem}{Theorem}[section]
\newtheorem{lemma}[theorem]{Lemma}
\newtheorem{proposition}[theorem]{Proposition}
\newtheorem{corollary}[theorem]{Corollary}

\newtheorem{problem}[theorem]{Problem}

\theoremstyle{definition}
\newtheorem{definition}[theorem]{Definition}

\newtheorem{algorithm}[theorem]{Algorithm}

\theoremstyle{remark}
\newtheorem{remark}[theorem]{Remark}

\numberwithin{theorem}{section}
\numberwithin{equation}{section}

\graphicspath{{./Figures/}}

\newcommand{\R}{\mathbb{R}}

\newcommand{\ve}{\varepsilon}

\newcommand{\id}{\mathrm{id}}

\newcommand{\rank}{\mathrm{rank}}

\newcommand{\qf}{\mathrm{qf}}

\newcommand{\Hess}{\mathrm{Hess}}

\begin{document}

\title{Adaptive Stochastic Gradient Descents on Manifolds with an Application on Weighted Low-Rank Approximation}

\author[1]{Peiqi Yang} 
\author[2]{Conglong Xu} 
\author[3]{Hao Wu}

\thanks{Correspondence to: Hao Wu (\texttt{haowu@gwu.edu})}

\thanks{The authors would like to thank Pierre-Antoine Absil and Yingfeng Hu for interesting and helpful conversations.}

\address{Department of Mathematics, The George Washington University, Phillips Hall, Room 739, 801 22nd Street, N.W., Washington DC 20052, USA. Telephone: 1-202-994-0653, Fax: 1-202-994-6760}

\email{pqyang@gwmail.gwu.edu, xuconglong@gwmail.gwu.edu, haowu@gwu.edu}



\begin{abstract}
We prove a convergence theorem for stochastic gradient descents on manifolds with adaptive learning rate and apply it to the weighted low-rank approximation problem.
\end{abstract}

\maketitle

\section{Introduction}\label{sec-intro}

\subsection{Stochastic Gradient Descents with Adaptive Learning Rates in $\R^n$} 
Let $\Omega$ be a probability space. A stochastic gradient descent in $\R^n$ for a $C^1$ cost function $F:\R^n\rightarrow \R$ is given by the iterative update
\begin{equation}\label{eq-iter-SGD-Rn}
\mathbf{x}_{t+1} = \mathbf{x}_t -\eta_t g(\mathbf{x}_t,\omega_t)
\end{equation}
where $g$ is a function $g:\R^n\times \Omega \rightarrow \R^n$ and $\{\omega_t\}_0^\infty \subset \Omega$ is a sequence of random variables of independent identical distribution in $\Omega$ satisfying $E(g(\mathbf{x},\omega_t)) = \nabla F (\mathbf{x})$ $\forall \mathbf{x} \in \R^n$.

In classical applications of stochastic gradient descents, the learning rates $\{\eta_t\}_0^\infty$ are usually a deterministic sequence of non-negative numbers satisfying 
\begin{equation}\label{eq-deterministic-assumption}
\sum_{t=0}^\infty \eta_t = \infty \text{ and } \sum_{t=0}^\infty \eta_t^2 < \infty. 
\end{equation}
More recent applications often adopt adaptive learning rates not satisfying these constraints. In \cite{Li-Orabona:2019}, Li and Orabona provided an elegant proof for the convergence of stochastic gradient descent \eqref{eq-iter-SGD-Rn} with the adaptive learning rates
\begin{equation}\label{eq-iter-SGD-Rn-ada-rate}
\eta_t = \frac{\alpha}{(\beta+\sum_{i=0}^{t-1} \|g(\mathbf{x}_i,\omega_i)\|^2)^{\frac{1}{2}+\ve}},
\end{equation}
where $\alpha$ and $\beta$ are positive numbers and $0 < \ve \le \frac{1}{2}$.

\subsection{Overview and Main Contributions} 
In the current paper, we generalize \cite{Li-Orabona:2019} to stochastic gradient descents on Riemannian manifolds, which leads to Theorem \ref{thm-Ada-SGD-mfd} below. While the core idea of our proof of Theorem \ref{thm-Ada-SGD-mfd} is similar to that in \cite{Li-Orabona:2019}, we need to overcome technical hurdles rooted in three issues:
\begin{enumerate}
	\item It is assumed in \cite{Li-Orabona:2019} that the cost function and its gradient are both globally Lipschitz. This limits the application of the results of \cite{Li-Orabona:2019}. We replace the assumption of global Lipschitzness by a notion of retraction-specific local Lipschitzness. See Definition \ref{def-R-lipschitz} below.
	\item To use the local Lipschitzness, we need to confine the search steps in a compact subset of the manifold. For this purpose, we opt to assume that the random search admits a ``$\kappa$-confinement". See Definition \ref{def-kappa-confinement} below. This also guarantees that the search sequence contains a converging subsequence.
	\item To prove Theorem \ref{thm-Ada-SGD-mfd}, we need to estimate the difference of the gradients of two closely related functions in Lemma \ref{lemma-gradient-difference}. When working on an Euclidean space, these two gradients are identical, which eliminates the need for such an estimation. But, when working on a Riemannian manifold, it took us some effort to control the difference between them.
\end{enumerate}

After resolving these issues, we apply the arguments similar to those in \cite{Li-Orabona:2019} to complete our proof of Theorem \ref{thm-Ada-SGD-mfd}. 

Then, as an application, we apply Theorem \ref{thm-Ada-SGD-mfd} to establish an algorithm for a Regularized Weighted Low-Rank Approximation Problem (Problem \ref{prob-RWLRA-1-reform}). 
We introduced in \cite{Xu-Yang-Wu-CSGD} a deterministic/``semi-adaptive" approach to stochastic gradient descents on manifolds. 
We will compare the computational efficiencies of the adaptive and semi-adaptive approaches when applied to Problem \ref{prob-RWLRA-1-reform}.

\section{A Convergent Theorem}\label{sec-Ada-SGD-mfd}
\subsection{Relevant Concepts}\label{subsec-concepts}

Before stating our main convergence theorem, we need to review some necessary technical concepts and their basic properties. In this manuscript, we only discuss the type of random functions given in Definition \ref{def-random-function}.

\begin{definition}\label{def-random-function}
Let $M$ be a differentiable manifold, and $\Omega$ a probability space. A function $f:M\times \Omega \rightarrow \R$ is called a random function on $M$. We say that $f$ is $k$th-order differentiable in the $M$-direction if, for every $\omega \in \Omega$, $f(\ast,\omega):M \rightarrow \R$ is $k$th-order differentiable. We also say that $f$ is locally bounded on $M$ if, for every compact set $K\subset M$, there is a $C_K>0$ such that $|f(x,\omega)|\leq C_K$ for every $(x,\omega)  \in K \times \Omega$.
\end{definition}

In Definition \ref{def-random-function}, $M$ is the manifold in which we search for the optimal parameters of the learning model. $\Omega$ is the space of samples used to guide our search. 

\begin{definition}\cite[Definition 4.1.1]{AMS}\label{def-retraction} 
Let $M$ be a differentiable manifold. A retraction on $M$ is a $C^1$ map $R:TM \rightarrow M$ such that, for every $x \in M$, the restriction $R_x=R|_{T_x M}$ satisfies
\begin{itemize}
	\item $R_x(\mathbf{0}_x)=x$, where $\mathbf{0}_x$ is the zero vector in $T_x M$,
	\item $dR_x(\mathbf{0}_x)=\id_{T_x M}$ under the canonical identification $T_{\mathbf{0}_x}T_x M \cong T_x M$, where $dR_x$ is the differential of $R_x$ and $\id_{T_x M}$ is the identity map of $T_x M$.
\end{itemize}
\end{definition}

We also need the following concept of retraction-specific Lipschitzness.

\begin{definition}\label{def-R-lipschitz}
Let $M$ be a Riemannian manifold, $R$ a given retraction on $M$, and $\Omega$ a probability space.  Suppose that the random function $f:M \times \Omega \rightarrow \R$ is first-order differentiable in the $M$-direction. Then, for each $x \in M$ and each $\omega \in \Omega$, the function $f_{x,\omega}:=f(R_x(\ast),\omega):T_x M \rightarrow \R$ is a first-order differentiable function. Its gradient $\nabla f_{x,\omega}$ is the vector in $T_x M$ dual to the differential $df_{x,\omega}$ via the Riemannian inner product $\left\langle \ast,\ast\right\rangle_x$.

\begin{itemize}
	\item  $f$ is said to have $R$-Lipschitz gradient in the $M$-direction if there is a constant $C>0$ such that 
	\[
	\|\nabla f_{x,\omega}(\mathbf{v}) - \nabla f_{x,\omega}(\mathbf{0}_x)\|_x \leq C\|\mathbf{v}\|_x
	\]
	for every $x \in M$, $\mathbf{v} \in T_x M$ and $\omega \in \Omega$.
	\item  $f$ is said to have locally $R$-Lipschitz gradient in the $M$-direction if, for every compact subset $K$ of $M$ and every $r>0$, there is a constant $C_{K,r}>0$ such that 
	\[
	\|\nabla f_{x,\omega}(\mathbf{v}) - \nabla f_{x,\omega}(\mathbf{0}_x)\|_x \leq C_{K,r}\|\mathbf{v}\|_x
	\]
	for every $x \in K$, every $\mathbf{v} \in T_x M$ satisfying $\|\mathbf{v}\|_x\leq r$ and every $\omega \in \Omega$.
\end{itemize}
In the case when $\Omega=\{\omega\}$ is a probability space of a single point, the above gives the definitions of the gradient of a deterministic function on $M$ being $R$-Lipschitz and locally $R$-Lipschitz.
\end{definition}

\begin{remark}\label{remark-gradient-coincide}
In Definition \ref{def-R-lipschitz}, denote by $\nabla_M f$ the gradient of $f$ with respenct to $M$, that is, $\nabla_M f(x,\omega) = \nabla f_w(x)$, where $f_\omega := f(\ast,\omega):M \rightarrow \R$. 
Then according to \cite[Remark 2.4]{Xu-Yang-Wu-CSGD}, we have that
\begin{equation}\label{eq-gradient-coincide}
\nabla f_{x,\omega}(\mathbf{0}_x) = \nabla_M f(x,\omega).
\end{equation} 
\end{remark}

To guarantee that there is a compact subset of the manifold that contains all steps of the stochastic gradient descent, we assume that, for some $\kappa>0$, there is a $\kappa$-confinement for the random function used in the stochastic gradient descent. 

\begin{definition}[$\kappa$-confinement]     \label{def-kappa-confinement}
Assume that: 
\begin{enumerate}
	\item $M$ is a Riemannian manifold with a fixed $C^2$ retraction $R:TM\rightarrow M$;
	\item $\Omega$ is a probability space;
	\item $f:M\times \Omega \rightarrow \R$ is a random function on $M$ that is first-order differentiable in the $M$-direction.
\end{enumerate}
For a fixed $\kappa>0$, a  $\kappa$-confinement of $f$ on $M$ is $C^2$ differentiable function $\rho: M \rightarrow \R$ satisfying:
\begin{itemize}
	\item For every $c\in \R$, the set $\{x \in M~\big{|}~ \rho(x)\leq c\}$ is compact;
	\item There exist $\rho_0,\rho_1\in \R$ such that 
	\begin{itemize}
	\item $\rho_0<\rho_1$, 
	\item	\begin{equation}\label{eq-def-confinement-1}
  \rho(R_{x}(-s\nabla_M f(x,\omega)))\leq \rho_1
	\end{equation}
	for every $s \in [0,\kappa]$, every $\omega \in \Omega$ and every $x \in M$ satisfying $\rho(x)\leq \rho_0$,
	\item  \begin{equation}\label{eq-def-confinement-2}
	\left\langle \nabla \rho (x), \nabla_M f(x,\omega) \right\rangle_x \geq \frac{\kappa}{2} \Big{|} \Hess(\rho \circ R_x)|_{-s \nabla_M f(x,\omega)}(\nabla_M f(x,\omega), \nabla_M f(x,\omega))\Big{|} 
	\end{equation}
	for every $s \in [0,\kappa]$, every $\omega \in \Omega$ and every $x \in M$ satisfying $\rho_0\leq \rho(x)\leq \rho_1$, where $\Hess(\rho\circ R_x)$ is the Hessian of the function $\rho\circ R_x:T_x M \rightarrow \R$, which is defined on the inner product space $T_x M$.
	\end{itemize}
\end{itemize}
\end{definition}

The following lemma explains the origin of the idea of $\kappa$-confinements. 

\begin{lemma}\label{lemma-confinement-construction}
Assume that: 
\begin{enumerate}
	\item $M$ is a Riemannian manifold with a fixed $C^2$ retraction $R:TM\rightarrow M$.
	\item $\Omega$ is a probability space.
	\item $f:M\times \Omega \rightarrow \R$ is a random function on $M$ that is first-order differentiable in the $M$-direction.
	\item $\widetilde{\kappa}>0$.
	\item $\rho: M \rightarrow \R$ is $C^2$ function satisfying:
\begin{itemize}
	\item For every $c\in \R$, the set $\{x \in M~\big{|}~ \rho(x)\leq c\}$ is compact;
	\item There exists a $\rho_0\in \R$ and a $\zeta>0$ such that $\left\langle \nabla \rho (x), \nabla_M f(x,\omega) \right\rangle_x \geq \zeta$ for every $\omega \in \Omega$ and every $x \in M$ satisfying $\rho(x)\geq \rho_0$.
	\end{itemize}
\end{enumerate}
Let 
\begin{enumerate}[(i)]
	\item $K_0  :=  \{x\in M ~\big{|}~ \rho(x) \leq \rho_0\}$, $G_0  :=    \max \{\|\nabla_M f(x,\omega)\|_x ~\big{|}~ x \in K_0, ~\omega \in \Omega\}$,
	\item\label{item-rho-1} \[
\rho_1  :=    \max \{\rho_0 + \widetilde{\kappa} |\left\langle \nabla \rho(x),\mathbf{v}\right\rangle_x| + \frac{1}{2}\widetilde{\kappa}^2|\Hess (\rho \circ R_x)|_{-s\mathbf{v}}(\mathbf{v},\mathbf{v})|~\big{|}~ x \in K_0, ~\mathbf{v} \in T_xM, ~\|\mathbf{v}\|_x \leq G_0, ~ s \in [0,\widetilde{\kappa}] \}, 
  \]
	\item $K_1  :=  \{x\in M ~\big{|}~ \rho(x) \leq \rho_1\}$ and $\widetilde{G}_1  :=    \max \{\|\nabla_M f(x,\omega)\|_x ~\big{|}~ x \in K_1, ~\omega \in \Omega\}$,
	\item \[
\kappa := \min\left\{\widetilde{\kappa}, \frac{2\zeta}{\max\{|\Hess (\rho \circ R_x)|_{-s\mathbf{v}}(\mathbf{v},\mathbf{v})|~\big{|}~x \in K_1, ~\mathbf{v} \in T_xM, ~\|\mathbf{v}\|_x \leq G_1, ~ s \in [0,\widetilde{\kappa}]\}}\right\}
\]
\end{enumerate}
Then $\rho$ is a $\kappa$-confinement of $f$.
\end{lemma}

\begin{proof}
By Lemma \ref{lemma-Taylor} below, for any $x\in M$, and any $s \geq 0$, we have that
\[
\rho(R_{x}(-s\nabla_M f(x,\omega))) = \rho(x) -s\left\langle \nabla \rho(x), \nabla_M f(x,\omega)\right\rangle + \frac{1}{2} s^2 \Hess(\rho\circ R_x)|_{-s^\star \nabla_M f(x,\omega)} (\nabla_M f(x,\omega),\nabla_M f(x,\omega)) 
\]
for some $s^\star \in [0,s]$. By the definition of $\rho_1$ in item (\ref{item-rho-1}) above, one can see that $\rho(R_{x}(-s\nabla_M f(x,\omega))) \leq \rho_1$ for any $x \in K_0$ and $s\in[0, \kappa] \subset[0, \widetilde{\kappa}]$. This proves Inequality \eqref{eq-def-confinement-1}.

For every $s \in [0,\kappa]$, every $\omega \in \Omega$ and every $x \in M$ satisfying $\rho_0\leq \rho(x)\leq \rho_1$, we have that
\[
\left\langle \nabla \rho (x), \nabla_M f(x,\omega) \right\rangle_x \geq \zeta \geq  \frac{\kappa}{2} \big{|} \Hess(\rho \circ R_x)|_{-s \nabla_M f(x,\omega)}(\nabla_M f(x,\omega), \nabla_M f(x,\omega))\big{|} .
\]
This proves Inequality \eqref{eq-def-confinement-2}.
\end{proof}

\subsection{The Convergent Theorem}\label{subsec-main-thm}

The following is the main theoretical result of the current paper.

\begin{theorem}\label{thm-Ada-SGD-mfd}
Assume that:
\begin{enumerate}
	\item $M$ is an $m$-dimensional Riemannian manifold equipped with a $C^2$ retraction $R:TM\rightarrow M$. 
	Also, $\left\langle \ast,\ast \right\rangle_x$ is the Riemannian inner product on $T_x M$ for all $x \in M $ and induces the norm $\|\ast\|_x$ on $T_x M$ for all $x \in M$. 
	\item \label{assumption-probability-space} $\Omega$ is a probability space with a probability measure $\mu$. 
	\item \label{assumption-random-function} $f:M \times \Omega \rightarrow \R$ is a random function on $M$ satisfying:
	\begin{enumerate}[(i)]
	  \item $f$ is first-order differentiable in the $M$-direction,
	  \item $f$ and $\|\nabla_M f\|$ are locally bounded on $M$,
		\item $f$ has locally $R$-Lipschitz gradient in the $M$-direction.
	\end{enumerate}
	\item $F:M\rightarrow \R$ is the expectation of $f(x,\omega)$ with respect to $\omega$. That is, $F(x)=E(f(x,\omega))=\int_\Omega f(x,\omega)d\mu$ for all $x\in M$.
	\item \label{assumption-confinement-function} For a fixed $\kappa>0$, $\rho:M\rightarrow \R$ is a $C^2$ $\kappa$-confinement of $f$. Fix $\rho_0, \rho_1\in \R$ for which Inequalities \eqref{eq-def-confinement-1} and \eqref{eq-def-confinement-2} are true. Let $K_1$ be the compact subsets of $M$ given by
	\begin{equation} \label{eq-K-1-def}
  K_1  :=  \left\{x\in M ~\big{|}~ \rho(x) \leq \rho_1 \right\}.
	\end{equation}
	\item $\{\omega_t\}_{t=0}^\infty$ is a sequence of independent random variables taking values in $\Omega$ with identical probability distribution $\mu$. 
	\item \label{assumption-initial value} $x_0 \in K_1$ is a fixed point.
	\item \label{assumption-constants} Fix positive constants $\ve, \alpha, \beta$ satisfying:
  \begin{enumerate}[(i)]
	\item $0<\ve\leq \frac{1}{2}$.
	\item \begin{equation}\label{eq-bound-kappa}
        \frac{\alpha}{\beta^{\frac{1}{2}+\ve}} \leq \kappa.
	\end{equation}
  \end{enumerate}
\end{enumerate}

Define a sequences $\{x_t\}_{t=0}^\infty$ of random elements of $M$ and a sequence $\{\eta_t\}_{t=0}^\infty$ of random positive numbers by 
\begin{equation}\label{eq-def-x-t}
\begin{cases}
x_{t+1} = R_{x_t}\left(-\eta_t \nabla_M f(x_t,\omega_t)\right) \\
\eta_t = \frac{\alpha}{(\beta+\sum_{s=0}^{t-1} \|\nabla_M f(x_s,\omega_s)\|_{x_s}^2)^{\frac{1}{2}+\ve}}
\end{cases}
\text{ for } t\geq 0.
\end{equation}
 Then:
\begin{itemize}
  \item $\{x_t\}_{t=0}^\infty$ is contained in the compact subset $K_1$ of $M$. In particular, it has convergent subsequences.
	\item $\{\|\nabla F(x_t)\|_{x_t}\}_{t=0}^\infty$ converges almost surely to $0$.
	\item Any limit point of $\{x_t\}_{t=0}^\infty$ is almost surely a stationary point of $F$.
\end{itemize}
\end{theorem}

Some of our technical lemmas are fairly basic. Some of these have already appeared in \cite{Bertsekas-Tsitsiklis:2000,Li-Orabona:2019,Xu-Yang-Wu-CSGD} for example. For the convenience of the reader, we nevertheless include these lemmas and the proofs of most of them here. The main outlier is Lemma \ref{lemma-gradient-difference}. Over Euclidean spaces, this lemma is trivial. So it does not appear in \cite{Bertsekas-Tsitsiklis:2000,Li-Orabona:2019}. Even in \cite{Xu-Yang-Wu-CSGD}, we only need a weaker version of this lemma, which is easier to prove.  See \cite[Lemma A.7]{Xu-Yang-Wu-CSGD}.

\begin{lemma}\label{lemma-Taylor}
Let $h:M\rightarrow\R$ be any $C^2$ function. Then, for any $x\in M$ and $\mathbf{v}\in T_x M$, there exists an $s^\star \in [0,1]$ such that $h(R_x(\mathbf{v})) = h(x) + \left\langle  \nabla h(x), \mathbf{v}\right\rangle_x +\frac{1}{2}\Hess(h\circ R_x)|_{s^\star \mathbf{v}} (\mathbf{v},\mathbf{v})$, where $\Hess(h\circ R_x)$ is the Hessian of the function $h\circ R_x:T_x M \rightarrow \R$, which is defined on the inner product space $T_x M$. 
\end{lemma}

\begin{proof}
Define $p(s)=h(R_x(s\mathbf{v}))$. Then, by Taylor's Theorem,  there exists an $s^\star \in [0,1]$ such that $p(1) = p(0) + p'(0) + \frac{1}{2}p''(s^\star)$. Note that:
\begin{itemize}
	\item $p(0) = h(R_x(\mathbf{0}_x)) =h(x)$ and $p(1) = h(R_x(\mathbf{v}))$;
	\item $p'(s) = \left\langle \nabla (h\circ R_x)(s\mathbf{v}) ,\mathbf{v}\right\rangle_x$ and therefore $p'(0) = \left\langle \nabla (h\circ R_x)(\mathbf{0}_x) ,\mathbf{v}\right\rangle_x=\left\langle \nabla h(x) ,\mathbf{v}\right\rangle_x$, where the last equation follows from Equation \eqref{eq-gradient-coincide};
	\item $p''(s) = \frac{d}{ds}\left\langle \nabla (h\circ R_x)(s\mathbf{v}) ,\mathbf{v}\right\rangle_x = \left\langle \frac{d}{ds}\nabla (h\circ R_x)(s\mathbf{v}) ,\mathbf{v}\right\rangle_x = \Hess(h\circ R_x)|_{s \mathbf{v}} (\mathbf{v},\mathbf{v})$.
\end{itemize}
Combining the above, we get that $h(R_x(\mathbf{v})) = h(x) + \left\langle  \nabla h(x), \mathbf{v}\right\rangle_x +\frac{1}{2}\Hess(h\circ R_x)|_{s^\star \mathbf{v}} (\mathbf{v},\mathbf{v})$.
\end{proof}

\begin{lemma}\label{lemma-x-t-confined}
$\rho(x_t) \leq \rho_1$ for all $t\geq 0$. That is, the sequence $\{x_t\}_{t=0}^\infty$ is contained in the compact subset $K_1$ of $M$ defined in Equation \eqref{eq-K-1-def}.
\end{lemma}

\begin{proof}
We prove this lemma by induction. For $t=0$, $\rho(x_0) \leq \rho_1$ by our choice. Assume that $x_t \in K_1$. Let us consider $x_{t+1}$. Note that 
\begin{equation}\label{eq-eta-t-alpha-beta-bound}
0< \eta_t \leq  \frac{\alpha}{\beta^{\frac{1}{2}+\ve}} \leq \kappa \text{ for all } t \geq 0
\end{equation}
and let $K_0 := \left\{x\in M ~\big{|}~ \rho(x) \leq \rho_0 \right\}$.

If $x_t \in K_0\subset K_1$, then, by Inequality \eqref{eq-def-confinement-1}, we have that 
\[ 
\rho(x_{t+1})= R_{x_t}\left(-\eta_t \nabla_M f(x_t,\omega_t)\right) \leq \rho_1.
\]
So $x_{t+1} \in K_1$.

Now consider the case $x_t \in K_1\setminus K_0\subset K_1$. By Lemma \ref{lemma-Taylor}, we have that
\begin{eqnarray*}
	&& \rho(x_{t+1}) = \rho (R_{x_t}(-\eta_t\nabla_M f(x_t,\omega_t))) \\ 
	& = & \rho (x_t) - \eta_t\left\langle  \nabla \rho(x_t), \nabla_M f(x_t,\omega_t)\right\rangle_x  + \frac{1}{2}\eta_t^2\Hess(\rho\circ R_x)|_{-s^\star \eta_t\nabla_M f(x_t,\omega_t)} (\nabla_M f(x_t,\omega_t),\nabla_M f(x_t,\omega_t)) 
\end{eqnarray*}
for some $s^\star \in [0,1]$. By Inequalities \eqref{eq-def-confinement-2} and \eqref{eq-eta-t-alpha-beta-bound}, we have that
\begin{eqnarray*}
&& \rho(x_{t+1}) \\
& = & \rho (x_t) - \eta_t\left\langle  \nabla \rho(x_t), \nabla_M f(x_t,\omega_t)\right\rangle_x  + \frac{1}{2}\eta_t^2\Hess(\rho\circ R_x)|_{-s^\star \eta_t\nabla_M f(x_t,\omega_t)} (\nabla_M f(x_t,\omega_t),\nabla_M f(x_t,\omega_t)) \\
& \leq & \rho (x_t)  +\eta_t\left(-\left\langle  \nabla \rho(x_t), \nabla_M f(x_t,\omega_t)\right\rangle_x  + \frac{1}{2}\kappa\big{|}\Hess(\rho\circ R_x)|_{-s^\star \eta_t\nabla_M f(x_t,\omega_t)} (\nabla_M f(x_t,\omega_t),\nabla_M f(x_t,\omega_t))\big{|}\right) \\
& \leq & \rho (x_t) \leq \rho_1.
\end{eqnarray*}
Again, $x_{t+1} \in K_1$.

This shows that $\{x_t\}_{t=0}^\infty \subset K_1$.
\end{proof}

\begin{lemma}\label{lemma-F-locally-Lipschitz}
$F$ is first-order differentiable and has locally $R$-Lipschitz gradient. Moreover, there exists a $C_1>0$ such that $\|\nabla (F\circ R_x)(\mathbf{v}) - \nabla F(x)\|_x \leq C_1\|\mathbf{v}\|_x$ and $F(R_x(\mathbf{v})) \leq F(x) + \left\langle \nabla F(x), \mathbf{v} \right\rangle_x + \frac{C_1}{2}\|\mathbf{v}\|_x^2$ for every $x\in K_1$ and every $\mathbf{v}\in T_x M$ satisfying $\|\mathbf{v}\|_x \leq \kappa G_1$, where $G_1$ is given by
\begin{equation} \label{eq-G-1-def}
G_1  :=    \max \{\|\nabla_M f(x,\omega)\|_x ~\big{|}~ x \in K_1, ~\omega \in \Omega\}.
\end{equation}
\end{lemma}

\begin{proof}
Since $\|\nabla_M f(x,\omega)\|$ is locally bounded, it follows from the Donimated Convergence Theorem that $F$ is first-order differentiable and $\nabla F(x) = \int_\Omega \nabla_M f(x,\omega)d\mu$. Since $R$ is $C^2$, $\|\nabla_{T_x M} f(R_x(\mathbf{v}),\omega)\|$ is also locally bounded, So, again by the Donimated Convergence Theorem, we have $\nabla (F\circ R)(\mathbf{v}) = \int_\Omega \nabla_{T_x M} f(R_x(\mathbf{v}),\omega)d\mu = \int_\Omega \nabla f_{x,\omega}(\mathbf{v})d\mu$, where $f_{x,\omega}:=f(R_x(\ast),\omega):T_x M \rightarrow \R$ is defined in Definition \ref{def-R-lipschitz}. 

Since $f$ has locally $R$-Lipschitz gradient, we know that for every compact subset $K$ of $M$ and every $r>0$, there is a constant $C_{K,r}>0$ such that $\|\nabla f_{x,\omega}(\mathbf{v}) - \nabla f_{x,\omega}(\mathbf{0}_x)\|_x \leq C_{K,r}\|\mathbf{v}\|_x$ for every $x \in K$, every $\mathbf{v} \in T_x M$ satisfying $\|\mathbf{v}\|_x\leq r$ and every $\omega \in \Omega$. Thus,
\begin{eqnarray*}
&& \|\nabla (F\circ R_x)(\mathbf{v}) - \nabla (F\circ R_x)(\mathbf{0}_x)\|_x = \|\int_\Omega \nabla f_{x,\omega}(\mathbf{v})d\mu - \int_\Omega \nabla f_{x,\omega}((\mathbf{0}_x))d\mu\|_x \\
& \leq & \int_\Omega\| \nabla f_{x,\omega}(\mathbf{v})-\nabla f_{x,\omega}((\mathbf{0}_x))\|_xd\mu \leq \int_\Omega C_{K,r}\|\mathbf{v}\|_x d\mu = C_{K,r}\|\mathbf{v}\|_x.
\end{eqnarray*}
This shows that $F$ has locally $R$-Lipschitz gradient.

Now consider the special case $K=K_1=\{x\in M~\big{|}~ \rho(x)\leq \rho_1\}$ and $r=\kappa G_1$. Let $C_1= C_{K_1,\kappa G_1}$. Then, by Equation \eqref{eq-gradient-coincide} and the above inequality, we have that $\|\nabla (F\circ R_x)(\mathbf{v}) - \nabla F(x)\|_x \leq C_1\|\mathbf{v}\|_x$ for every $x\in M$ satisfying $\rho(x)\leq \rho_1$ and every $\mathbf{v}\in T_x M$ satisfying $\|\mathbf{v}\|\leq \kappa G_1$. Let $p(s)=F(R_x(s\mathbf{v}))$ for $s\in [0,1]$. Then $p'(s)=\left\langle  \nabla (F\circ R_x)(s\mathbf{v}), \mathbf{v}\right\rangle_x$. And
\begin{eqnarray*}
F(R_x(\mathbf{v})) - F(x) & = & p(1) -p(0) = \int_0^1 p'(s) ds = \int_0^1 \left\langle  \nabla (F\circ R_x)(s\mathbf{v}), \mathbf{v}\right\rangle_x ds \\
& = & \int_0^1 \left\langle  \nabla F(x)+\nabla (F\circ R_x)(s\mathbf{v})-\nabla F(x), \mathbf{v}\right\rangle_x ds \\
& = &  \left\langle  \nabla F(x), \mathbf{v}\right\rangle_x + \int_0^1 \left\langle  \nabla (F\circ R_x)(s\mathbf{v})-\nabla F(x), \mathbf{v}\right\rangle_x ds \\
& \leq & \left\langle  \nabla F(x), \mathbf{v}\right\rangle_x + \int_0^1 \|\nabla (F\circ R_x)(s\mathbf{v})-\nabla F(x)\|_x \|\mathbf{v}\|_x ds \\
& \leq & \left\langle  \nabla F(x), \mathbf{v}\right\rangle_x + \int_0^1 C_1s \|\mathbf{v}\|_x^2 ds  = \left\langle  \nabla F(x), \mathbf{v}\right\rangle_x + \frac{C_1}{2}\|\mathbf{v}\|_x^2
\end{eqnarray*}
for every $x\in K_1$ and every $\mathbf{v}\in T_x M$ satisfying $\|\mathbf{v}\|\leq \kappa G_1$. This completes the proof.
\end{proof}

To proceed further, we need to understand the dependence of the solution of an ODE initial value problem on the initial value and other parameters.

\begin{lemma}\label{lemma-ode-C-1-dependence}\cite[Chapter 1, Proposition 6.1]{Taylor}
Let $h:\R \times \R^n \times \R^l \rightarrow \R^n$ be a $C^1$ function. Assume that there is an open set $U\subset \R^n \times \R^l$  and an interval $ (t_{-1},t_1)$ such that  $t_{-1}<t_0<t_1$ and the initial value problem 
\begin{equation}\label{eq-ode-ivp-parameters}
\frac{d\vec{y}}{dt} = h(t, \vec{y}, \vec{z}) \text{ and } \vec{y}(t_0) = \vec{x}
\end{equation}
has a solution $\vec{y}(t,\vec{x},\vec{z})$ for $t \in (t_{-1},t_1)$ and $\left[
\begin{array}{c}
\vec{x} \\
\vec{z}
\end{array}
\right] \in U$. Then, for each $t \in (t_{-1},t_1)$, the function $\vec{y}(t,\vec{x},\vec{z})$ is $C^1$ jointly in the initial value $\vec{x}$ and the parameter vector $\vec{z}$ for $\left[
\begin{array}{c}
\vec{x} \\
\vec{z}
\end{array}
\right] \in U$.
\end{lemma}

\begin{proof}
Define a $C^1$ function $H:\R \times \R^n \times \R^l \rightarrow \R \times \R^n \times \R^l$ by 
\[
H \left( \left[
\begin{array}{c}
\tau \\
\vec{y} \\
\vec{z}
\end{array}
\right] \right) = 
\left[
\begin{array}{c}
1 \\
h(\tau, \vec{y}, \vec{z}) \\
\vec{0}
\end{array}
\right].
\]
Then $\left[
\begin{array}{c}
t \\
\vec{y}(t,\vec{x},\vec{z}) \\
\vec{z}
\end{array}
\right]$ is the solution of the initial value problem
\[
\frac{d}{dt}\left[
\begin{array}{c}
\tau(t) \\
\vec{y}(t) \\
\vec{z}(t)
\end{array}
\right] =
H \left( \left[
\begin{array}{c}
\tau(t) \\
\vec{y}(t) \\
\vec{z}(t)
\end{array}
\right] \right) 
\text{ and } 
\left[
\begin{array}{c}
\tau(t_0) \\
\vec{y}(t_0) \\
\vec{z}(t_0)
\end{array}
\right] =
\left[
\begin{array}{c}
t_0 \\
\vec{x} \\
\vec{z}
\end{array}
\right]
\]
on the interval $(t_{-1},t_1)$. By \cite[Chapter 1, Proposition 6.1]{Taylor}, for each $t \in (t_{-1},t_1)$, the function $\left[
\begin{array}{c}
t \\
\vec{y}(t,\vec{x},\vec{z}) \\
\vec{z}
\end{array}
\right]$ is $C^1$ in the initial value $\left[
\begin{array}{c}
t_0 \\
\vec{x} \\
\vec{z}
\end{array}
\right]$. This proves the lemma.
\end{proof}

\begin{definition}\label{def-PT-R}
Recall that $TM \oplus TM$ is the vector bundle over $M$ obtained by direct summing $TM$ to itself. Its fiber over $x \in M$ is $T_x M \oplus T_x M$. It comes with a standard smooth manifold structure. For any $(\mathbf{u}_x, \mathbf{v}_x) \in T_x M \oplus T_x M$, define $PT_R(\mathbf{u}_x, \mathbf{v}_x) \in T_{R_x(\mathbf{u}_x)} M$ to be vector resulted from the parallel transportation of $\mathbf{v}_x$ along the curve $R_x(t\mathbf{u}_x)$ from $t=0$ to $t=1$. Then $PT_R: TM \oplus TM \rightarrow TM$ is a function.
\end{definition}

\begin{lemma}\label{lemma-PT-R-C-1}
The function $PT_R: TM \oplus TM \rightarrow TM$ defined in Definition \ref{def-PT-R} is $C^1$. Moreover, for any $x \in M$ and $\mathbf{u}_x, \mathbf{v}_x^1, \mathbf{v}_x^2 \in T_x M$, we have that
\begin{equation}\label{eq-PT-R-parallel}
\left\langle PT_R(\mathbf{u}_x, \mathbf{v}_x^1), PT_R(\mathbf{u}_x, \mathbf{v}_x^2)\right\rangle_{R_x(\mathbf{u}_x)} = \left\langle \mathbf{v}_x^1, \mathbf{v}_x^2\right\rangle_{x}.
\end{equation}
\end{lemma}

\begin{proof}
Equation \eqref{eq-PT-R-parallel} is a standard property of parallel transportation. We only need to prove that $PT_R: TM \oplus TM \rightarrow TM$ is $C^1$. Fix an $x \in M$ and $\mathbf{u}_x \in T_x M$. Consider the curve $\mathscr{C}$ given by $R_x(t\mathbf{u}_x)$ for $0\leq t \leq 1$. By the Lebesgue Number Lemma, there is a positive integer $p$ such that, for $l=1,2,\dots,p$, there is a coordinate chart $U_l$ such that the segment of $\mathscr{C}$ with $\frac{l-1}{p} \leq t \leq \frac{l}{p}$ is contained in $U_l$. By the continuity of $R$ and the compactness of closed intervals, we know that there is a neighborhood $\mathcal{W}$ of $\mathbf{u}_x$ in $T M$ such that, for any $\mathbf{w}_y\in \mathcal{W}$, the curve $R_y(t\mathbf{w}_y)$ with $\frac{l-1}{p} \leq t \leq \frac{l}{p}$ is contained in $U_l$.

Next, we prove that $PT_R$ is $C^1$ on the open subset $\widetilde{\mathcal{W}}:= \{(\mathbf{w}_y, \mathbf{v}_y) \in TM \oplus TM ~\big{|}~  \mathbf{w}_y\in \mathcal{W}\}$ of $TM \oplus TM$. We do so by inductively prove that $PT_R(\frac{l}{p}\mathbf{w}_y, \mathbf{v}_y)$ is $C^1$ in $(\mathbf{w}_y, \mathbf{v}_y)$ for $(\mathbf{w}_y, \mathbf{v}_y) \in \widetilde{\mathcal{W}}$ for $l=0,1,2, \dots,p$.

For $l=0$, $PT_R(\mathbf{0}_y, \mathbf{v}_y)=\mathbf{v}_y$ and is clearly $C^\infty$. Now assume that $PT_R(\frac{l-1}{p}\mathbf{w}_y, \mathbf{v}_y)$ is $C^1$ in $(\mathbf{w}_y, \mathbf{v}_y)$ for $(\mathbf{w}_y, \mathbf{v}_y) \in \widetilde{\mathcal{W}}$. We prove that $PT_R(\frac{l}{p}\mathbf{w}_y, \mathbf{v}_y)$ is $C^1$ in $(\mathbf{w}_y, \mathbf{v}_y)$ for $(\mathbf{w}_y, \mathbf{v}_y) \in \widetilde{\mathcal{W}}$. Denote by $(z_1,\dots,z_m)$ the coordinate functions for the chart $U_l$. These equip $TU_l$ with a frame $[\frac{\partial}{\partial z_1},\dots,\frac{\partial}{\partial z_m}]$. Note that $PT_R(t\mathbf{w}_y, \mathbf{v}_y)$ for $\frac{l-1}{p} \leq t \leq \frac{l}{p}$ is a parallel vector field along the curve $R_y(t\mathbf{w}_y)$ with $\frac{l-1}{p} \leq t \leq \frac{l}{p}$. So it satisfies the first order differential equation $\frac{D}{\partial t}\left(PT_R(t\mathbf{w}_y, \mathbf{v}_y)\right) = 0$ for $\frac{l-1}{p} \leq t \leq \frac{l}{p}$. Under the frame $[\frac{\partial}{\partial z_1},\dots,\frac{\partial}{\partial z_m}]$, this is 
\begin{equation}\label{eq-parallel-transport}
\frac{\partial v_k}{\partial t}  + \sum_{i=1}^m \sum_{j=1}^m \Gamma_{i,j}^k|_{R_y(t\mathbf{w}_y)}  \left(\frac{\partial}{\partial t} \left(z_i|_{R_y(t\mathbf{w}_y)}\right) \right)  v_j = 0 \text{ for } k =1,2,\dots,m,
\end{equation}
where $PT_R(t\mathbf{w}_y, \mathbf{v}_y) = \sum_{j=1}^m v_j \frac{\partial}{\partial z_j}$ and $\{\Gamma_{i,j}^k\}$ are the Christoffel symbols of the Levi-Civita connection with respected to the coordinates $(z_1,\dots,z_m)$, which are $C^{\infty}$ in $(z_1,\dots,z_m)$. Since $R$ is $C^2$, we know that:
\begin{itemize}
	\item $\Gamma_{i,j}^k|_{R_y(t\mathbf{w}_y)}$ is $C^2$ in $\mathbf{w}_y$ for $\mathbf{w}_y \in \mathcal{W}$,
	\item $\frac{\partial}{\partial t} \left(z_i|_{R_y(t\mathbf{w}_y)}\right)$ is $C^1$ in $\mathbf{w}_y$ for $\mathbf{w}_y \in \mathcal{W}$.
\end{itemize}

Write $\vec{v} = [v_1,\dots,v_m]^T \in \R^m$ and $\Lambda = [\sum_{i=1}^m \Gamma_{i,j}^k|_{R_y(t\mathbf{w}_y)}  \left(\frac{\partial}{\partial t} \left(z_i|_{R_y(t\mathbf{w}_y)}\right)\right)]_{m \times m}$. Then $\Lambda$ is a matrix of functions that are $C^1$ in $\mathbf{w}_y$ for $\mathbf{w}_y \in \mathcal{W}$. And Equation \eqref{eq-parallel-transport} becomes
\begin{equation}\label{eq-parallel-transport-1}
\frac{\partial \vec{v}}{\partial t}  + \Lambda  \vec{v} = 0.
\end{equation}
By Lemma \ref{lemma-ode-C-1-dependence}, the solution $\vec{v}$ of Equation \eqref{eq-parallel-transport-1} is $C^1$ jointly in $\mathbf{w}_y$ for $\mathbf{w}_y \in \mathcal{W}$ and in the initial value $\vec{v}|_{t=\frac{l-1}{p}}$. But $PT_R(\frac{l-1}{p}\mathbf{w}_y, \mathbf{v}_y)$ is $C^1$ in $(\mathbf{w}_y, \mathbf{v}_y)$ for $(\mathbf{w}_y, \mathbf{v}_y) \in \widetilde{\mathcal{W}}$ and $PT_R(\frac{l-1}{p}\mathbf{w}_y, \mathbf{v}_y) = [\frac{\partial}{\partial z_1},\dots,\frac{\partial}{\partial z_m}]\vec{v}|_{t=\frac{l-1}{p}}$. So $\vec{v}|_{t=\frac{l-1}{p}}$ is a column of functions that are $C^1$ in $(\mathbf{w}_y, \mathbf{v}_y)$ for $(\mathbf{w}_y, \mathbf{v}_y) \in \widetilde{\mathcal{W}}$ . Putting these together, we have that $PT_R(\frac{l}{p}\mathbf{w}_y, \mathbf{v}_y) = [\frac{\partial}{\partial z_1},\dots,\frac{\partial}{\partial z_m}]\vec{v}|_{t=\frac{l}{p}}$ is also $C^1$ in $(\mathbf{w}_y, \mathbf{v}_y)$ for $(\mathbf{w}_y, \mathbf{v}_y) \in \widetilde{\mathcal{W}}$.

This completes the induction and proves that $PT_R(\mathbf{w}_y, \mathbf{v}_y)$ is $C^1$ in $(\mathbf{w}_y, \mathbf{v}_y)$ for $(\mathbf{w}_y, \mathbf{v}_y) \in \widetilde{\mathcal{W}}$. In summary, we have established that, for every $(\mathbf{u}_x, \mathbf{v}_x) \in TM \oplus TM$, there is an open neighborhood $\widetilde{\mathcal{W}}$ of $(\mathbf{u}_x, \mathbf{v}_x)$ in $TM \oplus TM$ on which $PT_R$ is $C^1$. This proves that $PT_R: TM \oplus TM \rightarrow TM$ is $C^1$.
\end{proof}

\begin{lemma}\label{lemma-gradient-difference-local}
Fix any $x \in M$ and any $r>0$. Then there are an open neighborhood $U_x$ of $x$ in $M$ and a positive number $C_{x,r}$ such that 
\[
\left|\|\nabla F (R_y(\mathbf{u}))\|_{R_y(\mathbf{u})} - \|\nabla(F\circ R_y)(\mathbf{u})\|_y\right| \leq C_{x,r} \|\mathbf{u}\|_y
\] 
for every $y \in U_x$ and every $\mathbf{u} \in T_y M$ satisfying $\|\mathbf{u}\|_y \leq r$.
\end{lemma}

\begin{proof}
For any $y \in M$ and $\mathbf{u},\mathbf{v}\in T_y M$, we have that 
\begin{eqnarray*}
&&\left\langle \nabla (F\circ R_y)(\mathbf{u}), \mathbf{v}\right\rangle_y = d(F\circ R_y)|_{\mathbf{u}}(\mathbf{v}) = ((dF|_{R_y(\mathbf{u})})\circ (dR_y|_{\mathbf{u}}))(\mathbf{v})  \\
& = & \left\langle (\nabla F)(R_y(\mathbf{u})), (dR_y|_{\mathbf{u}})(\mathbf{v})\right\rangle_{R_y(\mathbf{u})} = \left\langle \mathrm{adj}(dR_y|_{\mathbf{u}})((\nabla F)(R_y(\mathbf{u}))), \mathbf{v}\right\rangle_{y},
\end{eqnarray*}
where $\mathrm{adj}(dR_y|_{\mathbf{u}})$ is the adjoint of $dR_y|_{\mathbf{u}}$ with respect to the inner products $\left\langle \ast,\ast\right\rangle_y$ and $\left\langle \ast,\ast\right\rangle_{R_y(\mathbf{u})}$. This shows that 
\begin{equation}\label{eq-lemma-gradient-difference-1}
\nabla (F\circ R_y)(\mathbf{u})=\mathrm{adj}(dR_y|_{\mathbf{u}})((\nabla F)(R_y(\mathbf{u}))).
\end{equation} 

Let $(W_x,(y_1,\dots,y_m))$ be a $C^\infty$ coordinate chart containing $x$. Fix an open neighborhood $U_x$ of $x$ such that its closure $\overline{U_x}$ is compact and contained in $W_x$. For $y\in W_x$, apply the Gram-Schmidt Process to the basis $[\frac{\partial }{\partial y_1},\dots,\frac{\partial }{\partial y_m}]$ of $T_x M$ and then normalize the lengths. This gives an orthonormal basis $[\mathbf{v}_1(y),\dots,\mathbf{v}_m(y)]$ for $T_y M$ and a $C^\infty$ $m\times m$ invertible upper-triangular matrix $P(y)$ such that $[\frac{\partial }{\partial y_1},\dots,\frac{\partial }{\partial y_m}]=[\mathbf{v}_1(y),\dots,\mathbf{v}_m(y)]P(y)$. For $y \in U_x$ and $\mathbf{u}\in T_y M$, we use the orthonormal basis $[\mathbf{v}_1(y),\dots,\mathbf{v}_m(y)]$ for $T_y M$ and the basis \linebreak $[PT_R(\mathbf{u},\mathbf{v}_1(y)),\dots,PT_R(\mathbf{u},\mathbf{v}_m(y))]$ for $T_{R_y(\mathbf{u})} M$. By Lemma \ref{lemma-PT-R-C-1}, $[PT_R(\mathbf{u},\mathbf{v}_1(y)),\dots,PT_R(\mathbf{u},\mathbf{v}_m(y))]$ is orthonormal and $C^1$ jointly in $y$ and $\mathbf{u}$. 

Under the standard identification $T_{\mathbf{u}}(T_y M) = T_y M$, denote by $[dR_{i,j}]$ the $m \times m$ matrix of the differential map $dR_y|_{\mathbf{u}}: T_y M \rightarrow T_{R_y(\mathbf{u})} M$ under these bases for the two tangent spaces. Then each $dR_{i,j}$ is $C^1$ jointly in $y$ and $\mathbf{u}$. Therefore, the gradient $\nabla_{\mathbf{u}} dR_{i,j}(y,\mathbf{u})$ is continuous jointly in $y$ and $\mathbf{u}$ for $1 \leq i,j \leq m$. So $D_{x,r} := \max\{\|\nabla_\mathbf{u} dR_{i,j}(y,\mathbf{u})\|_y ~\big{|}~1\leq i,j\leq m,~y\in \overline{U_x},~\mathbf{u}\in T_y M,~\|\mathbf{u}\|_y \leq r\}< \infty$. By the Mean Value Theorem, there is an $s^\star \in [0,1]$ such that 
\[
| dR_{i,j}(y,\mathbf{u}) - dR_{i,j}(y,\mathbf{0}_y)| = |\left\langle \nabla_{\mathbf{u}} dR_{i,j}(y,s^\star\mathbf{u}),\mathbf{u}\right\rangle_y| 
 \leq  \|\nabla_{\mathbf{u}} dR_{i,j}(y,s^\star\mathbf{u})\|_y  \|\mathbf{u}\|_y\leq D_{x,r} \|\mathbf{u}\|_y.
\]
Recall that $dR_y|_{\mathbf{0}_y} =\id_{T_y M}$ and $[PT_R(\mathbf{0},\mathbf{v}_1(y)),\dots,PT_R(\mathbf{0},\mathbf{v}_m(y))] = [\mathbf{v}_1(y),\dots,\mathbf{v}_m(y)]$. So $[dR_{i,j}(y,\mathbf{0}_y)]=I_m$, that is,
\[
dR_{i,j}(y,\mathbf{0}_y) =\delta_{i,j}=
\begin{cases}
1 &\text{if }i=j, \\ 
0 &\text{if }i\neq j.
\end{cases}
\]
Combining the above, we get that, for $y\in \overline{U_x}$ and $\mathbf{u}\in T_y M$ satisfying $\|\mathbf{u}\|_y \leq r$,
\begin{equation}\label{eq-lemma-gradient-difference-2}
|dR_{i,j}(y,\mathbf{u}) - \delta_{i,j}| \leq D_{x,r} \|\mathbf{u}\|_y.
\end{equation}
Furthermore, note that the matrix representing $\mathrm{adj} (dR_y|_{\mathbf{u}}): T_{R_y(\mathbf{u})} M \rightarrow T_y M$ is $[dR_{i,j}]^T$ under the bases $[PT_R(\mathbf{u},\mathbf{v}_1(y)),\dots,PT_R(\mathbf{u},\mathbf{v}_m(y))]$ for $T_{R_y(\mathbf{u})} M$ and $[\mathbf{v}_1(y),\dots,\mathbf{v}_m(y)]$ for $T_y M$ since both bases are orthonormal.

Denote by $\|\ast\|$ the standard norm on $R^m$ and by $\|\ast\|_F$ the Frobenius norm on $\R^{m \times m}$. Also, for $y\in \overline{U_x}$ and $\mathbf{u}\in T_y M$ satisfying $\|\mathbf{u}\|_y \leq r$, let $\vec{w} = \vec{w}(y,\mathbf{u})=[w_1(y,\mathbf{u}),\dots,w_m(y,\mathbf{u})]^T \in \R^m$ be such that
\[
\nabla F(R_y(\mathbf{u})) = [PT_R(\mathbf{u},\mathbf{v}_1(y)),\dots,PT_R(\mathbf{u},\mathbf{v}_m(y))] \vec{w}.
\]
Then, for $y\in \overline{U_x}$ and $\mathbf{u}\in T_y M$ satisfying $\|\mathbf{u}\|_y \leq r$, we have
\begin{eqnarray*}
&& \left|\|\nabla F (R_y(\mathbf{u}))\|_{R_y(\mathbf{u})} - \|\nabla(F\circ R_y)(\mathbf{u})\|_y\right| = \left|\|\nabla F (R_y(\mathbf{u}))\|_{R_y(\mathbf{u})} - \|\mathrm{adj}(R_y|_\mathbf{u})(\nabla F (R_y(\mathbf{u})))\|_y\right| \\
& = & \left| \|[PT_R(\mathbf{u},\mathbf{v}_1(y)),\dots,PT_R(\mathbf{u},\mathbf{v}_m(y))]\vec{w}\|_{R_y(\mathbf{u})} - \|[\mathbf{v}_1(y),\dots, \mathbf{v}_m(y)][dR_{i,j}(y,\mathbf{u})]^T\vec{w}\|_{y} \right| \\
& = & \left| \|\vec{w}\| - \|[dR_{i,j}(y,\mathbf{v})]^T\vec{w}\| \right| \leq \|[dR_{i,j}(y,\mathbf{u})]^T\vec{w} - \vec{w}\|= \|([dR_{i,j}(y,\mathbf{u})-\delta_{i,j}]^T)\vec{w} \| \\
& = & \sqrt{\sum_{j=1}^m (\sum_{i=1}^m  ( dR_{i,j}(y,\mathbf{u})-\delta_{i,j}) w_i(y,\mathbf{u}))^2} \leq \sqrt{m\sum_{j=1}^m \sum_{i=1}^m  ( dR_{i,j}(y,\mathbf{u})-\delta_{i,j})^2 (w_i(y,\mathbf{u}))^2} \\
& \leq & \sqrt{m\sum_{j=1}^m \sum_{i=1}^m  D_{x,r}^2\|\mathbf{u}\|_y^2 (w_i(y,\mathbf{u})))^2} = \sqrt{m^2 D_{x,r}^2\|\mathbf{u}\|_y^2 \sum_{i=1}^m (w_i(y,\mathbf{u}))^2} \\
& = & m D_{x,r}\|\mathbf{u}\|_y\|\vec{w}\| = m D_{x,r} \|\mathbf{u}\|_y \|\nabla F (R_y(\mathbf{u}))\|_{R_y(\mathbf{u})} \leq m D_{x,r} \|\mathbf{u}\|_y G_{x,r},
\end{eqnarray*}
where $G_{x,r} := \max\{ \|\nabla F(R_y(\mathbf{u}))\|_{R_y(\mathbf{u})}~\big{|}~ y \in \overline{U_x}, ~\mathbf{u}\in T_y M, ~ \|\mathbf{u}\|\leq  r \}< \infty$.
Now define $C_{x,r} := m D_{x,r} G_{x,r}>0$. The above computations show that Lemma \ref{lemma-gradient-difference-local} is true for this choice of $C_{x,r}$.
\end{proof}

The following lemma is an improved version of \cite[Lemma A.7]{Xu-Yang-Wu-CSGD} and follows easily from Lemma \ref{lemma-gradient-difference-local}. 

\begin{lemma} \label{lemma-gradient-difference}
For any fixed positive number $r$, there is a constant $C_2>0$ such that 
\[
\left|\|\nabla F (R_x(\mathbf{v}))\|_{R_x(\mathbf{v})} - \|\nabla(F\circ R_x)(\mathbf{v})\|_x\right| \leq C_2 \|\mathbf{v}\|_x
\] 
for every $x \in K_1$ and every $\mathbf{v} \in T_x M$ satisfying $\|\mathbf{v}\|_x \leq r$.
\end{lemma}

\begin{proof}
By Lemma \ref{lemma-gradient-difference-local}, for every $x \in K_1$, there are an open neighborhood $U_x$ of $x$ in $M$ and a positive number $C_{x,r}$ such that 
\[
\left|\|\nabla F (R_y(\mathbf{u}))\|_{R_y(\mathbf{u})} - \|\nabla(F\circ R_y)(\mathbf{u})\|_y\right| \leq C_{x,r} \|\mathbf{u}\|_y
\] 
for every $y \in U_x$ and every $\mathbf{u} \in T_y M$ satisfying $\|\mathbf{u}\|_y \leq r$. But $K_1$ is compact. So there are $x_1,\cdots x_n \in K_1$ such that $K_1 \subset \bigcup_{i=1}^n U_{x_i}$. Set $C_2:= \max\{C_{x_1,r}.\dots.C_{x_n,r}\}$. Then Lemma \ref{lemma-gradient-difference} is true for this choice of $C_2$.
\end{proof}

With the above lemmas in hand, the rest of the proof of Theorem \ref{thm-Ada-SGD-mfd} proceeds similarly to that of \cite[Theorem 1]{Li-Orabona:2019}. The next lemma is a manifold version of \cite[Lemma 3]{Li-Orabona:2019}. Similar lemmas have been used in the previous literature (for example \cite{Bertsekas-Tsitsiklis:2000}).

\begin{lemma}\label{lemma-bound-F-gradient-sum}
Set $F^{\ast}= \min\{F(x)~\big{|}~x\in K_1\}$. Then, for any $T\geq 1$, 
\begin{equation}\label{eq-bound-F-gradient-sum}
E(\sum_{t=0}^T \eta_t \|\nabla F(x_t)\|_{x_t}^2) \leq F(x_0) - F^{\ast} + \frac{C_1}{2}E(\sum_{t=0}^T \eta_t^2 \|\nabla_M f(x_t,\omega_t)\|_{x_t}^2),
\end{equation}
where $C_1$ is the positive constant from Lemma \ref{lemma-F-locally-Lipschitz}.
\end{lemma}

\begin{proof}
By Lemma \ref{lemma-F-locally-Lipschitz}, we have that
\[
F(x_{t+1}) = F(R_{x_t}(-\eta_t \nabla_M f(x_t,\omega_t)) \leq F(x_{t}) - \eta_t\left\langle \nabla F(x_t),  \nabla_M f(x_t,\omega_t)\right\rangle_{x_t} + \frac{C_1}{2} \eta_t^2\|\nabla_M f(x_t,\omega_t) \|_{x_t}^2.
\]
Taking expectations on both sides, we get that
\[
E(F(x_{t+1})) \leq E(F(x_{t})) - E(\eta_t\left\langle \nabla F(x_t),  \nabla_M f(x_t,\omega_t)\right\rangle_{x_t}) + \frac{C_1}{2} E(\eta_t^2\|\nabla_M f(x_t,\omega_t) \|_{x_t}^2)).
\]
Note that
\[
E(\eta_t\left\langle \nabla F(x_t),  \nabla_M f(x_t,\omega_t)\right\rangle_{x_t}) = E(E(\eta_t\left\langle \nabla F(x_t),  \nabla_M f(x_t,\omega_t)\right\rangle_{x_t}~\big{|}~ x_t,\eta_t)) = E(\eta_t\left\langle \nabla F(x_t),  \nabla F(x_t)\right\rangle_{x_t}).
\]
So 
\begin{equation}\label{eq-bound-F-gradient-sum-1}
E(\eta_t \|\nabla F(x_t)\|_{x_t}^2) \leq E(F(x_t)) - E(F(x_{t+1})) + \frac{C_1}{2}E( \|\nabla_M f(x_t,\omega_t)\|_{x_t}^2).
\end{equation}
Summing Inequality \eqref{eq-bound-F-gradient-sum-1} from $0$ to $T$, we get that
\begin{equation}\label{eq-bound-F-gradient-sum-2}
E(\sum_{t=0}^T \eta_t \|\nabla F(x_t)\|_{x_t}^2) \leq F(x_0) - E(F(x_{T+1})) + \frac{C_1}{2}E(\sum_{t=0}^T \eta_t^2 \|\nabla_M f(x_t,\omega_t)\|_{x_t}^2),
\end{equation}
where $E(F(x_0)) = F(x_0)$ since $x_0$ is fixed. But $E(F(x_{T+1}))\geq F^{\ast}$ since $\{x_{t}\}_{t=0}^\infty \subset K_1$. This prove Inequality \eqref{eq-bound-F-gradient-sum}.
\end{proof}

The next two lemmas are \cite[Lemma 1 and 2]{Li-Orabona:2019}. Please see \cite{Li-Orabona:2019} and the references therein for their proofs.

\begin{lemma}[] \label{lemma-Li-Orabona-1} 
{
\makeatletter
\newif\ifnobrackets
\renewcommand\@cite[2]{\ifnobrackets\else[\fi{#1\if@tempswa , #2\fi}\ifnobrackets\else]\fi\nobracketsfalse}
\newcommand\nbcite{\nobracketstrue\cite}
\makeatother

\textup{[\nbcite[Proposition 2]{Alber-Iusem-Solodov:1998}; \nbcite[Lemma A.5]{Mairal:2013}]} 
}
Let $\{a_t\}_{t=0}^\infty$ and $\{b_t\}_{t=0}^\infty$be two sequence of non-negative numbers. Assume that 
\begin{itemize}
	\item $\sum_{t=0}^\infty a_t b_t$ converges,
	\item $\sum_{t=0}^\infty a_t$ diverges,
	\item there exists an $L\geq 0$ such that $|b_{t+1}-b_t|\leq L a_t$ for $t\geq 0$. 
\end{itemize}
Then $\lim_{t\rightarrow \infty} b_t =0$.
\end{lemma}

\begin{lemma}\cite[Lemma 2]{Li-Orabona:2019}  \label{lemma-Li-Orabona-2}
Let $a_0>1$, $a_t\geq 0$ for $t=1,\dots, T$ and $b>1$. Then 
\[
\sum_{t=1}^T \frac{a_t}{(a_0+\sum_{i=1}^ta_i)^b} \leq \frac{1}{(b-1)a_0^{b-1}}.
\]
\end{lemma}

Now we are ready to adapt the proof of \cite[Theorem 1]{Li-Orabona:2019} to manifolds and prove Theorem \ref{thm-Ada-SGD-mfd}.

\begin{lemma}\label{lemma-gradient-square-sum-converge}
We have
\begin{itemize}
	\item $\sum_{t=0}^\infty \eta_t^2\| \nabla_M f(x_t, \omega_t)\|_{x_t}^2$ converges,
	\item $\sum_{t=0}^\infty \eta_t\| \nabla F(x_t)\|_{x_t}^2$ converges almost surely,
	\item $\sum_{t=0}^\infty \eta_t = \infty$.
\end{itemize}
\end{lemma}

\begin{proof}[Proof. (Following \cite{Li-Orabona:2019}.)]
\[
\sum_{t=0}^\infty \eta_t^2\| \nabla_M f(x_t, \omega_t)\|_{x_t}^2 = \sum_{t=0}^\infty \eta_{t+1}^2\| \nabla_M f(x_t, \omega_t)\|_{x_t}^2 + \sum_{t=0}^\infty (\eta_t^2-\eta_{t+1}^2)\| \nabla_M f(x_t, \omega_t)\|_{x_t}^2.
\]
By Lemma \ref{lemma-Li-Orabona-2}, for any $T\geq 1$,
\[
\sum_{t=0}^T \eta_{t+1}^2\| \nabla_M f(x_t, \omega_t)\|_{x_t}^2 = \sum_{t=0}^T \frac{\alpha^2 \| \nabla_M f(x_t, \omega_t)\|_{x_t}^2}{(\beta+\sum_{i=0}^{t} \|\nabla_M f(x_i,\omega_i)\|_{x_i}^2)^{1+2\ve}} \leq \frac{\alpha^2}{2\ve \beta^{2\ve}}.
\]
So 
\[
\sum_{t=0}^\infty \eta_{t+1}^2\| \nabla_M f(x_t, \omega_t)\|_{x_t}^2 \leq \frac{\alpha^2}{2\ve \beta^{2\ve}}.
\]
Note that $\{\eta_t\}$ is a decreasing sequence of positive numbers. 
Since $\{x_{t}\}_{t=0}^\infty \subset K_1$, we have 
\[
\sum_{t=0}^\infty (\eta_t^2-\eta_{t+1}^2)\| \nabla_M f(x_t, \omega_t)\|_{x_t}^2 \leq \sum_{t=0}^\infty (\eta_t^2-\eta_{t+1}^2) G_1^2\leq G_1^2 \eta_0^2 = G_1^2 \frac{\alpha^2}{\beta^{1+2\ve}},
\] 
where $G_1$ is given in Equation \eqref{eq-G-1-def}.
Thus, $\sum_{t=0}^\infty \eta_{t}^2\| \nabla_M f(x_t, \omega_t)\|_{x_t}^2 \leq \frac{\alpha^2}{2\ve \beta^{2\ve}}+ G_1^2 \frac{\alpha^2}{\beta^{1+2\ve}}$ and is therefore convergent.

By Lemma \ref{lemma-bound-F-gradient-sum}, we have that 
\begin{eqnarray*}
E(\sum_{t=0}^\infty \eta_t\| \nabla F(x_t)\|_{x_t}^2) & \leq &  F(x_0) - F^{\ast} + \frac{C_1}{2}E(\sum_{t=0}^\infty \eta_t^2 \|\nabla_M f(x_t,\omega_t)\|_{x_t}^2) \\
& \leq & F(x_0) - F^{\ast} + \frac{C_1}{2}(\frac{\alpha^2}{2\ve \beta^{2\ve}}+ G_1^2 \frac{\alpha^2}{\beta^{1+2\ve}})<\infty.
\end{eqnarray*}
This implies that the probability of $\sum_{t=0}^\infty \eta_t\| \nabla F(x_t)\|_{x_t}^2<\infty$ is $1$. In other words, $\sum_{t=0}^\infty \eta_t\| \nabla F(x_t)\|_{x_t}^2$ converges almost surely.

Finally, for the series  $\sum_{t=0}^\infty \eta_t$, we have that, since $\{x_{t}\}_{t=0}^\infty \subset K_1$,
\[
\sum_{t=0}^\infty \eta_t = \sum_{t=0}^\infty\frac{\alpha}{(\beta+\sum_{i=0}^{t-1} \|\nabla_M f(x_i,\omega_i)\|_{x_i}^2)^{\frac{1}{2}+\ve}} \geq  \sum_{t=0}^\infty\frac{\alpha}{(\beta+tG_1^2)^{\frac{1}{2}+\ve}} =\infty
\]
where $G_1$ is given in Equation \eqref{eq-G-1-def} and $\frac{1}{2}+\ve \leq 1$ by Assumption \eqref{assumption-constants} in Theorem \ref{thm-Ada-SGD-mfd}.
\end{proof}

\begin{proof}[Proof of Theorem \ref{thm-Ada-SGD-mfd}. (Following \cite{Li-Orabona:2019} mostly.)]
Let $G_1$ be given in Equation \eqref{eq-G-1-def}, and $C_1$, $C_2$ be the positive constants given in Lemmas \ref{lemma-F-locally-Lipschitz} and \ref{lemma-gradient-difference}, where we take $r=\frac{\alpha G_1}{\beta^{\frac{1}{2}+\ve}}$ in Lemma \ref{lemma-gradient-difference}. By Lemma \ref{lemma-x-t-confined}, we have
\begin{eqnarray*}
&& \left|\|\nabla F(x_{t+1})\|_{x_{t+1}}^2 -\|\nabla F(x_{t})\|_{x_{t}}^2\right| = (\|\nabla F(x_{t+1})\|_{x_{t+1}}+\|\nabla F(x_{t})\|_{x_{t}})~\big{|}\|\nabla F(x_{t+1})\|_{x_{t+1}} -\|\nabla F(x_{t})\|_{x_{t}}\big{|} \\
& \leq & 2G_1\big{|}\|\nabla F(x_{t+1})\|_{x_{t+1}} -\|\nabla F(x_{t})\|_{x_{t}}\big{|}\\
& = & 2G_1\big{|}\|\nabla F(R_{x_t}\left(-\eta_t \nabla_M f(x_t,\omega_t)\right))\|_{R_{x_t}\left(-\eta_t \nabla_M f(x_t,\omega_t)\right)} -\|\nabla F(x_{t})\|_{x_{t}}\big{|} \\
& \leq & 2G_1(\big{|}\|\nabla F(R_{x_t}\left(-\eta_t \nabla_M f(x_t,\omega_t)\right))\|_{R_{x_t}\left(-\eta_t \nabla_M f(x_t,\omega_t)\right)} -\|\nabla (F\circ R_{x_t})\left(-\eta_t \nabla_M f(x_t,\omega_t)\right)\|_{x_{t}}\big{|} \\
&& +  \big{|} \|\nabla (F\circ R_{x_t})\left(-\eta_t \nabla_M f(x_t,\omega_t)\right)\|_{x_{t}}  -\|\nabla F(x_{t})\|_{x_{t}}\big{|})
\end{eqnarray*}
By Lemmas \ref{lemma-x-t-confined} and \ref{lemma-gradient-difference}, 
\begin{eqnarray*}
&& \big{|}\|\nabla F(R_{x_t}\left(-\eta_t \nabla_M f(x_t,\omega_t)\right))\|_{R_{x_t}\left(-\eta_t \nabla_M f(x_t,\omega_t)\right)} -\|\nabla (F\circ R_{x_t})\left(-\eta_t \nabla_M f(x_t,\omega_t)\right)\|_{x_{t}}\big{|} \\
& \leq &  C_2 \eta_t  \|\nabla_M f(x_t,\omega_t)\|_{x_{t}} \leq C_2 G_1 \eta_t.
\end{eqnarray*}
Note that, by Lemma \ref{lemma-x-t-confined}, $x_t \in K_1$ for $t \geq 0$. Therefore, $\|\eta_t \nabla_M f(x_t,\omega_t)\|_{x_{t}} \le \frac{\alpha G_1}{\beta^{\frac{1}{2}+\ve}} \le \kappa G_1$. Consequencetly, by Lemma \ref{lemma-F-locally-Lipschitz}, we have 
\[
\big{|} \|\nabla (F\circ R_{x_t})\left(-\eta_t \nabla_M f(x_t,\omega_t)\right)\|_{x_{t}} -\|\nabla F(x_{t})\|_{x_{t}}\big{|} \leq C_1\eta_t  \|\nabla_M f(x_t,\omega_t)\|_{x_{t}}\leq C_1 G_1 \eta_t.
\]
So
\begin{equation}\label{eq-gradient-square-difference-bound}
\left|\|\nabla F(x_{t+1})\|_{x_{t+1}}^2 -\|\nabla F(x_{t})\|_{x_{t}}^2\right| \leq 2  G_1^2 (C_1+C_2) \eta_t.
\end{equation}
In summary, we have that
\begin{itemize}
	\item $\sum_{t=0}^\infty \eta_t\| \nabla F(x_t)\|_{x_t}^2$ converges almost surely by Lemma \ref{lemma-gradient-square-sum-converge},
	\item $\sum_{t=0}^\infty \eta_t = \infty$ by Lemma \ref{lemma-gradient-square-sum-converge},
	\item $\left|\|\nabla F(x_{t+1})\|_{x_{t+1}}^2 -\|\nabla F(x_{t})\|_{x_{t}}^2\right| \leq 2  G_1^2 (C_1+C_2) \eta_t$ by Inequality \eqref{eq-gradient-square-difference-bound}.
\end{itemize}
Thus, by Lemma \ref{lemma-Li-Orabona-1}, $\{\|\nabla F(x_t)\|_{x_t}\}_{t=0}^\infty$ converges almost surely to $0$. By Lemma \ref{lemma-x-t-confined}, $\{x_{t}\}_{t=0}^\infty \subset K_1$. So $\{x_{t}\}_{t=0}^\infty$ has convergent subsequences, which converge almost surely to stationary points of $F$.
\end{proof}

\section{An Application: Weighted Low-rank Approximation}\label{sec-WLRA-Ada-SGD}

\subsection{The Weighted Low-Rank Approximation Problem}

\begin{problem}[The Weighted Low-Rank Approximation Problem]\label{prob-WLRA}
Assume that:
\begin{enumerate}
    \item \label{problem-assumption-1} $m$, $n$ and $k$ are fixed positive integers satisfying $k \leq \min\{m,n\}$,
    \item \label{problem-assumption-2} $A=[a_{i,j}] \in \R^{m\times n}$ is a given matrix of constants,
    \item \label{problem-assumption-3} $W=[w_{i,j}] \in \R^{m\times n}$ is a given matrix of weights satisfying $w_{i,j} \geq 0$ for $(i,j) \in \{1,2,\dots,m\}\times \{1,2,\dots,n\}$ and $\sum_{i=1}^m \sum_{j=1}^n w_{i,j} = 1$,
\end{enumerate}
where $\R^{m\times n}$ is the space of $m \times n$ real matrices. Define $\hat{F}:\R^{m\times n} \rightarrow \R$ by \newline
$\hat{F}(P) = \sum_{i=1}^m \sum_{j=1}^n w_{i,j}(a_{i,j}-p_{i,j})^2$ for $P=[p_{i,j}] \in \R^{m\times n}$. Solve for
\begin{equation*}
\mathrm{argmin}\{\hat{F}(P) ~|~ P\in \R^{m\times n},~  \rank P \leq k\}.
\end{equation*}
\end{problem}

Problem \ref{prob-WLRA} and its special case, the Matrix Completion Problem, are studied extensively in the recent literature. (See, for example, \cite{Bertsimas-Cory-Lo-Pauphilet:2024,Yan-Tang-Li:2024,Yan-Zhang:2024,Lee:2024,Kelner-Li-Liu-Sidford-Tian:2023,Chakraborty-Dey:2023,Yang-Ma:2023,Bordenave-Coste-Nadakuditi:2023,Wan-Cheng:2023,Boumal-Absil:2011}). 

Problem \ref{prob-WLRA} is known to be NP-hard \cite{Gillis-Glineur:2011}. One existing approach to this problem is to reformulate it as an uncontrained problem on $\R^{m\times k}\times \R^{n\times k}$ and then estimate the solution to the uncontrained problem by various first order or second order algorithms on the Euclidean space. (See, for instance, \cite{Ban-Woodruff-Zhang:2019,Boumal-Absil:2011,Srebro-Jaakkola:2003}). We approach Problem \ref{prob-WLRA} by reformulating it as an unconstrained problem on a Riemmanian manifold in \cite{Xu-Yang-Wu-CSGD}.

In the current paper, we continue with this approach but using adaptive learning rates.
To guarantee the convergence of the algorithm, we consider a regularized version of Problem \ref{prob-WLRA} below.

\begin{problem}[A Regularized Weighted Low Rank Approximation Problem]\label{prob-RWLRA-1}
Under Assumptions (\ref{problem-assumption-1})-(\ref{problem-assumption-3}) in Problem \ref{prob-WLRA}, fix a positive number $\lambda$ and define $F:\R^{m\times n} \rightarrow \R$ by ${F(P) = \sum_{i=1}^m \sum_{j=1}^n w_{i,j}(a_{i,j}-p_{i,j})^2+\lambda \|P\|_F^2}$ for $P=[p_{i,j}] \in \R^{m\times n}$ where $\|P\|_F :=\sqrt{\sum_{i=1}^m \sum_{j=1}^n p_{i,j}^2}$ is the Frobenius norm. Solve for
\[
\mathrm{argmin}\{F(P) ~|~ P\in \R^{m\times n},~  \rank P \leq k\}.
\]
\end{problem}

To reformulate Problem \ref{prob-RWLRA-1} as an unconstrained problem, let us quickly review the Reduced Singular Value Decomposition. 
The $n\times k$ Stiefel manifold $V_k(\R^n)$ is defined by
\[
V_k(\R^n)  =  \{V \in \R^{n\times k}~|~ V^T V =I_k\},
\]
where $I_k$ is the $k \times k$ identity matrix. Define a function $D^{k\times k}_k: \R^k \rightarrow \R^{k\times k}$ by 
\begin{equation}\label{eq-def-D-k-by-k}
D^{k\times k}_k([x_1,\dots,x_k]^T)=[d_{i,j}], \text{ where } d_{i,j}=\begin{cases} x_i & \text{if } i=j, \\ 0 & \text{otherwise}. \end{cases}
\end{equation} 
For $P \in \R^{m\times n}$, $\rank P \leq k$ if any only if it admits the following Reduced Singular Value Decomposition
\begin{equation}\label{eq-rsvd}
P = U D^{k\times k}_k(\mathbf{x}) V^T,
\end{equation}
where $U \in V_k(\R^m)$, $V \in V_k(\R^n)$ and $\mathbf{x} \in \R^k$.

The Reduced Singular Value Decomposition motivates us to reformulate Problem \ref{prob-RWLRA-1} into Problem \ref{prob-RWLRA-1-reform}, which is an unconstrained problem over the Riemannian manifold $V_k(\R^m)\times \R^k \times V_k(\R^n)$. 
We then apply Theorem \ref{thm-Ada-SGD-mfd} to design a stochastic gradient descent with adaptive learning rates over this manifold for Problem \ref{prob-RWLRA-1-reform}. 

\begin{problem}[A Reformulated Regularized Weighted Low Rank Approximation Problem]\label{prob-RWLRA-1-reform}
Under Assumptions (\ref{problem-assumption-1})-(\ref{problem-assumption-3}) in Problem \ref{prob-WLRA}, fix a positive number $\lambda$ and let $F:\R^{m\times n} \rightarrow \R$ be as in Problem \ref{prob-RWLRA-1}. Solve for
\[
\mathrm{argmin}\{F(U D^{k\times k}_k(\mathbf{x}) V^T) ~|~ (U,\mathbf{x},V)  \in V_k(\R^m)\times \R^k \times V_k(\R^n) \}.
\]
\end{problem}
For $\mathbf{x} = [x_1,\dots,x_k]^T \in \R^k$ the standard Euclidean norm of $\mathbf{x}$ is given by $\|\mathbf{x}\| := \sqrt{\sum_{l=1}^k x_l^2}$. Note that, for $(U,\mathbf{x},V)  \in V_k(\R^m)\times \R^k \times V_k(\R^n)$ and $P=U D^{k\times k}_k(\mathbf{x}) V^T$,
\begin{equation}\label{eq-norm-equal}
\|P\|_F = \|\mathbf{x}\|.
\end{equation}

Next, we apply Theorem \ref{thm-Ada-SGD-mfd} to design a convergent stochastic gradient descent algorithm with adaptive learning rates for Problem \ref{prob-RWLRA-1-reform}. 

\subsection{The Random Function and the Retraction}

The basic setup here is similar to that in \cite{Xu-Yang-Wu-CSGD}. 
For the convenience of the reader, we still present it here.

Recall that we normalize the weight matrix $W=[w_{i,j}]$ so that $\sum_{i=1}^m\sum_{j=1}^n w_{i,j}=1$. 
\begin{definition} \label{def-mu-prob-space}
Define the probability space $\Omega_{m,n}$ by
\begin{equation}
\Omega_{m,n} = \{(i,j) \in \mathbb{Z}^2 ~|~  1 \le i \le m,\  1 \le j \le n \}.
\end{equation}
The probability measure $\mu$ on $\Omega_{m,n}$ is given by $\mu(\{(i,j)\})=w_{i,j}$ for every $(i,j) \in \Omega_{m,n}$.
\end{definition}

\begin{definition}\label{def-func-G}
For $F:\R^{m\times n} \rightarrow \R$ given in Problem \ref{prob-RWLRA-1-reform}, define $G:V_k(\R^m)\times \R^k \times V_k(\R^n)\rightarrow \R$ by
\begin{equation}\label{eq-def-G-RWLRA-1}
G(U,\mathbf{x},V) = F(U D^{k\times k}_k(\mathbf{x}) V^T),
\end{equation}
for all $(U,\mathbf{x},V) \in V_k(\R^m)\times \R^k \times V_k(\R^n)$. 
\end{definition}

\begin{definition}\label{def-random-functions-tau-gamma-RWLRA-1}
For $(\tau,\gamma)\in \Omega_{m,n}$, define $\hat{f}_{\tau,\gamma}:\R^{m\times n}\rightarrow \R$ and $f_{\tau,\gamma}:\R^{m\times n}\rightarrow \R$ by 
\begin{eqnarray}
\label{eq-def-hat-f-tau-gamma-RWLRA-1} \hat{f}_{\tau,\gamma}(P): & = & (a_{\tau,\gamma}-p_{\tau,\gamma})^2,\\
\label{eq-def-f-tau-gamma-RWLRA-1} f_{\tau,\gamma} (P): & = & \hat{f}_{\tau,\gamma} (P) + \lambda \|P\|_F^2 
= (a_{\tau,\gamma}-p_{\tau,\gamma})^2 + \lambda \|P\|_F^2,
\end{eqnarray}
for all $P=[p_{i,j}]\in \R^{m\times n}$, and define $g_{\tau,\gamma}:V_k(\R^m)\times \R^k \times V_k(\R^n)\rightarrow \R$ by
\begin{equation}\label{eq-def-hat-g-tau-gamma-RWLRA-1} 
g_{\tau,\gamma}(U,\mathbf{x},V): = f_{\tau,\gamma} (U D^{k\times k}_k(\mathbf{x}) V^T) = \hat{f}_{\tau,\gamma} (U D^{k\times k}_k(\mathbf{x}) V^T) + \lambda \| \mathbf{x} \|^2,
\end{equation}
for all $(U,\mathbf{x},V)  \in V_k(\R^m)\times \R^k \times V_k(\R^n)$, where $D^{k\times k}_k(\mathbf{x})$ is defined in Equation \eqref{eq-def-D-k-by-k}. 
The random function that we use is $g:V_k(\R^m)\times \R^k \times V_k(\R^n) \times \Omega_{m,n} \rightarrow \R$ given by $g(U,\mathbf{x},V;\tau,\gamma) = g_{\tau,\gamma}(U,\mathbf{x},V)$.
\end{definition}

\begin{lemma}\label{lemma-g-tau-gamma-RWLRA-1-expectation}
Let $G: V_k(\R^m)\times \R^k \times V_k(\R^n) \rightarrow \R$ be as in Definition \ref{def-func-G}. 
We take the expectation of $g$ over the probability space $\Omega_{m,n}$ with respect to the random input $(\tau, \gamma)$ with the probability distribution $\mu$ given in Definition \ref{def-mu-prob-space}.
For $(U,\mathbf{x},V) \in V_k(\R^m)\times \R^k \times V_k(\R^n)$, we have that the expectation of $g$ is
\begin{equation*}
E_{(\tau,\gamma)\sim \mu}(g_{\tau,\gamma}(U,\mathbf{x},V)) = G(U,\mathbf{x},V).
\end{equation*}
\end{lemma}

To define the retraction, we first describe the tangent spaces and a retraction of the Stiefel manifold in Lemma \ref{lemma-stiefel}.

\begin{definition}\label{def-qf}
Assume that $k\leq n$. For an $n\times k$ real matrix $C$ with linearly independent columns, define $\qf(C)=Q$ in the $QR$ decomposition $C=QR$, where 
\begin{itemize}
	\item $Q\in V_k(\R^n)$,
	\item $R$ is a $k \times k$ upper triangular matrix with positive entries along its diagonal.
\end{itemize}
Computationally, $\qf(C)$ can be obtained by applying the Gram-Schmidt Process on the columns of $C$, scaling the resulting orthogonal set of vectors into an orthonormal set and then using this orthonormal set of vectors as the columns of $\qf(C)$.
\end{definition}

\begin{lemma}\cite[Examples 3.5.2, 3.6.2 and 4.1.3]{AMS} \label{lemma-stiefel}
The $n \times k$ Stiefel manifold $V_k(\R^n)$ is a Riemannnian submanifold of $\R^{n\times k}$. Moreover:
\begin{itemize}
	\item For $X \in V_k(\R^n)$, the tangent space of $V_k(\R^n)$ at $X$ is $T_X V_k(\R^n) = \{ Z \in \R^{n\times k} ~|~X^TZ+Z^TX=0\}$.
	\item Denote by $\Pi_X$ the orthogonal projection $\Pi_X: \R^{n\times k} \rightarrow T_X V_k(\R^n)$. Then, for any $\xi \in \R^{n\times k}$, $\Pi_X(\xi)= (I_n -XX^T)\xi +\frac{1}{2}X(X^T\xi-\xi^TX)$.
	\item For any $X\in V_k(\R^n)$ and $Z \in T_X  V_k(\R^n)$, $X+Z\in \R^{n\times k}$ has linearly independent columns. Define $R^{V_k(\R^n)} :TV_k(\R^n)\rightarrow V_k(\R^n)$ by $R^{V_k(\R^n)}_X(Z) = \qf(X+Z)$. Then $R^{V_k(\R^n)}$ is a retraction on $V_k(\R^n)$.
\end{itemize}
\end{lemma}

The manifold $V_k(\R^m)\times \R^k \times V_k(\R^n)$ is a Riemannian submanifold of the Euclidean space $\R^{m\times k} \times \R^{k} \times \R^{n\times k}$. 
Note that the tangent space of the manifold $V_k(\R^m)\times \R^k \times V_k(\R^n)$ at $(U,\mathbf{x},V) \in V_k(\R^m)\times \R^k \times V_k(\R^n)$ is $T_U V_k(\R^m)\times \R^k \times T_VV_k(\R^n)$. We are ready to describe the retraction on the manifold $V_k(\R^m)\times \R^k \times V_k(\R^n)$.

\begin{definition}\label{def-retraction-GS-prod}
For $(U,\mathbf{x},V) \in V_k(\R^m)\times \R^k \times V_k(\R^n)$ and $(Y,\hat{\mathbf{x}},Z) \in  T_{U}V_k(\R^m)\times \R^k \times T_{V}V_k(\R^n)$, define $R_{(U,\mathbf{x},V)}(Y,\hat{\mathbf{x}},Z) = (\qf(U+Y),\mathbf{x} +\hat{\mathbf{x}}, \qf(V +Z))$. Then, by Lemma \ref{lemma-stiefel}, $R: TV_k(\R^m)\times \R^k \times TV_k(\R^n) \rightarrow V_k(\R^m)\times \R^k \times V_k(\R^n)$ is a retraction on $V_k(\R^m)\times \R^k \times V_k(\R^n)$.
\end{definition}

The following lemma is very useful in finding the gradient of the restriction of a function on a Riemannian submanifold.

\begin{lemma} \cite[Equation (3.37)]{AMS} \label{lemma-gradient-submfd}
Let $\overline{M}$ be a Riemannian manifold and $f: \overline{M} \rightarrow \R$ a differentiable function. Assume that $M$ is a Riemannian submanifold of $\overline{M}$. For any $x \in M$, denote by $\pi_x:T_x\overline{M} \rightarrow T_x M$ the orthogonal projection. Then $\nabla (f|_M)(x) = \pi_x(\nabla f(x))$ for any $x \in M$.
\end{lemma}

Now let $f: \R^{m\times n} \rightarrow \R$ be any differentiable function. Write 
\begin{itemize}
	\item $P=U D^{k\times k}_k(\mathbf{x}) V^T$ where $(U,\mathbf{x},V) \in \R^{m\times k} \times \R^{k} \times \R^{n\times k}$,
	\item $P=[p_{i,j}] \in \R^{m\times n}$, $U=[u_{i,j}] \in \R^{m\times k}$, $V=[v_{i,j}] \in \R^{n\times k}$ and $\mathbf{x}=[x_1,\dots,x_k]^T\in \R^k$.
\end{itemize}
Consider $f=f(P) =f(U D^{k\times k}_k(\mathbf{x}) V^T)$. Note that 
\begin{equation}\label{eq-p-uxv}
p_{i,j} = \sum_{l=1}^k u_{i,l}x_l v_{j,l}.
\end{equation} By the Chain Rule, we have that, for $l=1,\dots,k$, $i=1,\dots,m$ and $j=1,\dots,n$, 
\begin{equation}\label{eq-partial-dev}
\frac{\partial f}{\partial u_{i,l}} = \sum_{j=1}^n \frac{\partial f}{\partial p_{i,j}}x_l v_{j,l}, \hspace{1pc}\frac{\partial f}{\partial v_{j,l}} = \sum_{i=1}^m \frac{\partial f}{\partial p_{i,j}}x_l u_{i,l} \hspace{1pc} \text{and}\hspace{1pc} \frac{\partial f}{\partial x_{l}} = \sum_{i=1}^m \sum_{j=1}^n\frac{\partial f}{\partial p_{i,j}} u_{i,l}v_{j,l}.
\end{equation}

Let us compute the gradient of $g_{\eta,\gamma}$ next.

\begin{corollary}\label{cor-random-g-gradient-RWLRA-1}
Define the matrices
\begin{eqnarray*}
\nabla_U \hat{f}_{\tau,\gamma} & = & \left[\frac{\partial \hat{f}_{\tau,\gamma}}{\partial u_{i,l}}\right]_{m\times k}= \left[-2\delta_{\tau,i} (a_{\tau,\gamma}-p_{\tau,\gamma})x_l v_{\gamma,l}\right]_{m\times k}, \\
\nabla_V \hat{f}_{\tau,\gamma} & = & \left[\frac{\partial \hat{f}_{\tau,\gamma}}{\partial v_{j,l}}\right]_{n\times k} =  \left[-2\delta_{\gamma,j} (a_{\tau,\gamma}-p_{\tau,\gamma})x_l u_{\tau,l}\right]_{n\times k}, \\
\nabla_{\mathbf{x}} \hat{f}_{\tau,\gamma} & = & \left[\begin{array}{c}
\frac{\partial f_{\tau,\gamma}}{\partial x_{1}} \\
\frac{\partial f_{\tau,\gamma}}{\partial x_{2}} \\
\vdots \\
\frac{\partial f_{\tau,\gamma}}{\partial x_{k}}
\end{array}\right]
= \left[\begin{array}{c}
-2 (a_{\tau,\gamma}-p_{\tau,\gamma})u_{\tau,1} v_{\gamma,1} \\
-2 (a_{\tau,\gamma}-p_{\tau,\gamma})u_{\tau,2} v_{\gamma,2} \\
\vdots \\
-2 (a_{\tau,\gamma}-p_{\tau,\gamma})u_{\tau,k} v_{\gamma,k}
\end{array}\right]
\end{eqnarray*}
Then the gradient of $g_{\tau,\gamma}$ at any $(U,\mathbf{x},V) \in V_k(\R^m)\times \R^k \times V_k(\R^n)$ is 
\begin{equation}\label{eq-g-gradient-RWLRA-1}
\nabla g_{\tau,\gamma}(U,\mathbf{x},V) = (\Pi_U(\nabla_U \hat{f}_{\tau,\gamma}), \nabla_{\mathbf{x}} \hat{f}_{\tau,\gamma} + 2\lambda \mathbf{x}, \Pi_V(\nabla_V \hat{f}_{\tau,\gamma})) \in T_U V_k(\R^m)\times \R^k \times T_V V_k(\R^n),
\end{equation}
where $\Pi_X(\xi)= (I_n -XX^T)\xi +\frac{1}{2}X(X^T\xi-\xi^TX)$ is the orthogonal projection given in Lemma \ref{lemma-stiefel}.
\end{corollary}
\begin{proof}
For $(U,\mathbf{x},V)\in \R^{m\times k} \times \R^k \times \R^{n\times k}$, view $\hat{f}_{\tau,\gamma}$ as a composite function of $(U,\mathbf{x},V)$ via $\hat{f}_{\tau,\gamma}=\hat{f}_{\tau,\gamma}(P)$ and $P= U D^{k\times k}_k(\mathbf{x}) V^T$. Then 
\begin{eqnarray*}
\frac{\partial \hat{f}_{\tau,\gamma}}{\partial u_{i,l}} & = & -2\delta_{\tau,i} (a_{\tau,\gamma}-p_{\tau,\gamma})x_l v_{\gamma,l}, \\
\frac{\partial \hat{f}_{\tau,\gamma}}{\partial v_{j,l}} & = & -2\delta_{\gamma,j} (a_{\tau,\gamma}-p_{\tau,\gamma})x_l u_{\tau,l}, \\
\frac{\partial \hat{f}_{\tau,\gamma}}{\partial x_{l}} & = &
-2 (a_{\tau,\gamma}-p_{\tau,\gamma})u_{\tau,l} v_{\gamma,l}.
\end{eqnarray*}
Then, by Lemmas \ref{lemma-stiefel} and \ref{lemma-gradient-submfd}, the gradient of $g_{\tau,\gamma}$ is given by Equation \eqref{eq-g-gradient-RWLRA-1}.
\end{proof}

\subsection{The $\kappa$-confinement}

Next we give a $\kappa$-confinement for the random function $g$.

\begin{definition} \label{def-amax}
For $A=[a_{i,j}] \in \R^{m\times n}$ in Problem \ref{prob-RWLRA-1-reform}, define $a := \max \{a_{i,j}^2 \ | \ (i,j) \in \Omega_{m,n}\}$.
\end{definition}

\begin{definition} \label{def-rho}
The function $\rho : V_k(\R^m)\times \R^k \times V_k(\R^n)\rightarrow \R$ is defined by $\rho(U,\mathbf{x},V) = \| \mathbf{x} \|^2$ for any $(U,\mathbf{x},V) \in V_k(\R^m)\times \R^k \times V_k(\R^n)$.
\end{definition}

We show below that $\rho$ is a $\kappa$-confinement for the random function $g$ as defined in Definition \ref{def-kappa-confinement} with appropriately chosen $\kappa$, $\rho_0$ and $\rho_1$.

\begin{lemma}\cite[Lemma 3.12]{Xu-Yang-Wu-CSGD}  \label{lemma-rho-g-gradient}
Let $\rho: V_k(\R^m)\times \R^k \times V_k(\R^n)\rightarrow \R$ be the function given in Definition \ref{def-rho} and $g_{\tau,\gamma}: V_k(\R^m)\times \R^k \times V_k(\R^n)\rightarrow \R$ the function given in Definition \ref{def-random-functions-tau-gamma-RWLRA-1}. For any $(U,\mathbf{x},V)  \in V_k(\R^m)\times \R^k \times V_k(\R^n)$,
\begin{equation} \label{eq-rho-g-gradient}
\left\langle \nabla \rho(U,\mathbf{x},V), \nabla g_{\tau,\gamma}(U,\mathbf{x},V)\right\rangle = -4(a_{\tau,\gamma}-p_{\tau,\gamma})p_{\tau,\gamma} + 4\lambda \rho(U,\mathbf{x},V) \ge -a + 4\lambda \| \mathbf{x} \|^2.
\end{equation}
In particular, $\left\langle \nabla \rho(U,\mathbf{x},V), \nabla g_{\tau,\gamma}(U,\mathbf{x},V)\right\rangle \geq 0$ if $\rho(U,\mathbf{x},V) \geq \frac{a}{4\lambda}$.
\end{lemma}

\begin{lemma}\label{lemma-hessian}
Let $\rho: V_k(\R^m)\times \R^k \times V_k(\R^n)\rightarrow \R$ be the function given in Definition \ref{def-rho} and $R$ the retraction given in Definition \ref{def-retraction-GS-prod}. For any $(U,\mathbf{x},V)  \in V_k(\R^m)\times \R^k \times V_k(\R^n)$ and any $(Y',\hat{\mathbf{x}}',Z') \in  T_{U}V_k(\R^m)\times \R^k \times T_{V}V_k(\R^n)$, $ \Hess(\rho \circ R_{(U,\mathbf{x},V)})|_{(Y',\hat{\mathbf{x}}',Z')} $ is independent of $(Y',\hat{\mathbf{x}}',Z')$. Moreover,
\begin{equation}\label{eq-hessian}
\Hess(\rho \circ R_{(U,\mathbf{x},V)}) (\nabla g_{\tau,\gamma}(U,\mathbf{x},V) ,\nabla g_{\tau,\gamma}(U,\mathbf{x},V))
= 8\sum_{l=1}^k (-(a_{\tau,\gamma}-p_{\tau,\gamma})u_{\tau,l} v_{\gamma,l} + \lambda x_l)^2.
\end{equation}
\end{lemma}

\begin{proof}
Note that $(\rho \circ R_{(U,\mathbf{x},V)})(Y,\hat{\mathbf{x}},Z)=\|\mathbf{x}+\hat{\mathbf{x}}\|^2$ for any $(U,\mathbf{x},V)  \in V_k(\R^m)\times \R^k \times V_k(\R^n)$ and any $(Y,\hat{\mathbf{x}},Z) \in  T_{U}V_k(\R^m)\times \R^k \times T_{V}V_k(\R^n)$. So $\Hess(\rho \circ R_{(U,\mathbf{x},V)})|_{(Y',\hat{\mathbf{x}}',Z')}((Y,\hat{\mathbf{x}},Z),(Y,\hat{\mathbf{x}},Z))=2\|\hat{\mathbf{x}}\|^2$ for any $(U,\mathbf{x},V)  \in V_k(\R^m)\times \R^k \times V_k(\R^n)$ and any $(Y,\hat{\mathbf{x}},Z), ~(Y',\hat{\mathbf{x}}',Z') \in  T_{U}V_k(\R^m)\times \R^k \times T_{V}V_k(\R^n)$. In particular, it is independent of $(Y',\hat{\mathbf{x}}',Z')$. Thus by Equation \eqref{eq-g-gradient-RWLRA-1}, for any $(U,\mathbf{x},V)  \in V_k(\R^m)\times \R^k \times V_k(\R^n)$, we know that $\Hess(\rho \circ R_{(U,\mathbf{x},V)}) (\nabla g_{\tau,\gamma}(U,\mathbf{x},V) ,\nabla g_{\tau,\gamma}(U,\mathbf{x},V))= 2\|\nabla_{\mathbf{x}} g_{\tau,\gamma} \|^2 =2\|\nabla_{\mathbf{x}} \hat{f}_{\tau,\gamma} + 2\lambda\mathbf{x}\|^2$. Combining this with Corollary \ref{cor-random-g-gradient-RWLRA-1}, we get Equation \eqref{eq-hessian}.
\end{proof}

\begin{lemma} \label{lemma-rho0-condition}
Let $\rho: V_k(\R^m)\times \R^k \times V_k(\R^n)\rightarrow \R$ be the function given in Definition \ref{def-rho} and $R$ be the retraction given in Definition \ref{def-retraction-GS-prod}. 
We have that, for every $(U,\mathbf{x},V) \in V_k(\R^m)\times \R^k \times V_k(\R^n) $,
\begin{equation} \label{eq-kappa-conf-hess-con}
\left\langle \nabla \rho(U,\mathbf{x},V), \nabla g_{\tau,\gamma}(U,\mathbf{x},V)\right\rangle \ge \frac{\kappa}{2} | \Hess(\rho \circ R_{(U,\mathbf{x},V)}) (\nabla g_{\tau,\gamma}(U,\mathbf{x},V) ,\nabla g_{\tau,\gamma}(U,\mathbf{x},V)) |
\end{equation}
if the positive real number $\kappa$ satisfies 
\begin{equation} \label{eq-kappa-rho0-condition}
-a + 4\lambda \|\mathbf{x}\|^2 \ge 8\kappa \left(2ka + 2k \|\mathbf{x}\|^2 + \lambda^2 \|\mathbf{x}\|^2 \right).
\end{equation}
\end{lemma}

\begin{proof}
From Equation \eqref{eq-hessian}, we have
\begin{eqnarray} 
&&\Hess(\rho \circ R_{(U,\mathbf{x},V)}) (\nabla g_{\tau,\gamma}(U,\mathbf{x},V) ,\nabla g_{\tau,\gamma}(U,\mathbf{x},V)) \nonumber \\
& = & 8\sum_{l=1}^k (-(a_{\tau,\gamma}-p_{\tau,\gamma})u_{\tau,l} v_{\gamma,l} + \lambda x_l)^2 \le  8\sum_{l=1}^k \left( 2(a_{\tau,\gamma}-p_{\tau,\gamma})^2 u_{\tau,l}^2 v_{\gamma,l}^2 + 2\lambda^2 x_l^2 \right) \nonumber \\
& \le & 16\sum_{l=1}^k (a_{\tau,\gamma}-p_{\tau,\gamma})^2 + 16\lambda^2 \sum_{l=1}^k x_l^2 = 16\lambda^2 \| \mathbf{x} \|^2 + 16\sum_{l=1}^k (a_{\tau,\gamma}-p_{\tau,\gamma})^2 \nonumber \\
& \le & 16\lambda^2 \| \mathbf{x} \|^2 + 16\sum_{l=1}^k 2\left( a_{\tau,\gamma}^2 + p_{\tau,\gamma}^2 \right) \le 16\lambda^2 \| \mathbf{x} \|^2 + 16\sum_{l=1}^k 2\left( a + \| \mathbf{x} \|^2 \right) \nonumber \\
& = & 16 \left( 2ka + 2k \| \mathbf{x} \|^2 + \lambda^2 \| \mathbf{x} \|^2 \right). \label{eq-hess-upbd}
\end{eqnarray}
Then from Equation \eqref{eq-rho-g-gradient} and Inequality \eqref{eq-kappa-rho0-condition}, we have
\begin{eqnarray*}
&&\left\langle \nabla \rho(U,\mathbf{x},V), \nabla g_{\tau,\gamma}(U,\mathbf{x},V)\right\rangle \\
& = & -4(a_{\tau,\gamma}-p_{\tau,\gamma})p_{\tau,\gamma} + 4\lambda \| \mathbf{x} \|^2 \geq -a_{\tau,\gamma}^2 + 4\lambda \| \mathbf{x} \|^2 \\
& \ge & -a + 4\lambda \| \mathbf{x} \|^2 \ge 8\kappa \left( 2ka + 2k \| \mathbf{x} \|^2 + \lambda^2 \| \mathbf{x} \|^2 \right)\\
& \ge & \frac{\kappa}{2} | \Hess(\rho \circ R_{(U,\mathbf{x},V)}) (\nabla g_{\tau,\gamma}(U,\mathbf{x},V) ,\nabla g_{\tau,\gamma}(U,\mathbf{x},V)) |.
\end{eqnarray*}
\end{proof}

\begin{lemma} \label{lemma-rho1-condition}
Let $\rho: V_k(\R^m)\times \R^k \times V_k(\R^n)\rightarrow \R$ be the function given in Definition \ref{def-rho} and $R$ the retraction given in Definition \ref{def-retraction-GS-prod}. 
We have
\begin{equation} \label{eq-rho-R-le-rho1}
\rho(R_{(U,\mathbf{x},V)}(-s\nabla g_{\tau,\gamma}(U,\mathbf{x},V)))\leq \rho_1
\end{equation}
for every $(U,\mathbf{x},V) \in V_k(\R^m)\times \R^k \times V_k(\R^n) $ satisfying $ \|\mathbf{x}\|^2 \le \rho_0$ and for every $s \in [0,\kappa]$, if the positive real numbers $\rho_0$, $\rho_1$ and $\kappa$ satisfy that
\begin{equation} \label{eq-kappa-rho1-condition}
\rho_0 + a\kappa + (16ka + 8\lambda^2\rho_0 + 16k\rho_0)\kappa^2 \le \rho_1.
\end{equation}
\end{lemma}

\begin{proof}
By Lemma \ref{lemma-Taylor}, for any $(U,\mathbf{x},V)  \in V_k(\R^m)\times \R^k \times V_k(\R^n)$, and any $s \geq 0$, we have that
\begin{eqnarray*}
& \rho(R_{(U,\mathbf{x},V)}(-s\nabla g_{\tau,\gamma}(U,\mathbf{x},V))) = & \rho(U,\mathbf{x},V) -s\left\langle \nabla \rho(U,\mathbf{x},V), \nabla g_{\tau,\gamma}(U,\mathbf{x},V)\right\rangle \\
&& + \frac{s^2}{2}  \Hess(\rho \circ R_{(U,\mathbf{x},V)}) (\nabla g_{\tau,\gamma}(U,\mathbf{x},V) ,\nabla g_{\tau,\gamma}(U,\mathbf{x},V))) .
\end{eqnarray*}
From Equation \eqref{eq-hessian} and Inequality \eqref{eq-hess-upbd}, we have
\begin{equation} \label{eq-hess-est}
\Hess(\rho \circ R_{(U,\mathbf{x},V)}) (\nabla g_{\tau,\gamma}(U,\mathbf{x},V) ,\nabla g_{\tau,\gamma}(U,\mathbf{x},V)) \le 16 \left( 2ka + 2k \| \mathbf{x} \|^2 + \lambda^2 \| \mathbf{x} \|^2 \right).
\end{equation}
Then by Inequality \eqref{eq-rho-g-gradient}, for $ \| \mathbf{x} \|^2 \le \rho_0$ and $s \in [0,\kappa]$, we have
\begin{eqnarray*}
&& \rho(R_{(U,\mathbf{x},V)}(-s\nabla g_{\tau,\gamma}(U,\mathbf{x},V))) \\
& \le & \| \mathbf{x} \|^2 - s(-a + 4\lambda \| \mathbf{x} \|^2) + 8s^2 \left( 2ka + 2k \| \mathbf{x} \|^2 + \lambda^2 \| \mathbf{x} \|^2 \right) \\
& \le & \| \mathbf{x} \|^2 + as + \left(16ka + 8\lambda^2 \| \mathbf{x} \|^2 + 16k \| \mathbf{x} \|^2 \right)s^2 \\
& \le & \rho_0 + a\kappa + (16ka + 8\lambda^2\rho_0 + 16k\rho_0)\kappa^2 \le \rho_1.
\end{eqnarray*}
\end{proof}

\begin{proposition}\label{prop-rho0-rho1-value}
For any fixed $\lambda>0$, assume the positive real numbers $\rho_0$, $\rho_1$ and $\kappa$ satisfy Inequality  \eqref{eq-kappa-rho1-condition} above and the following inequalities
\begin{eqnarray}
& 0  <  \kappa  <  \frac{\lambda}{4k + 2\lambda^2},  \label{eq-kappa-lambda-constraint} \\
& \rho_0 \ge  \frac{(1+16k\kappa)a}{4\lambda - 16k\kappa - 8\lambda^2 \kappa}.  \label{eq-rho0-def} 
\end{eqnarray}
Then the function $\rho: V_k(\R^m)\times \R^k \times V_k(\R^n)\rightarrow \R$ given in Definition \ref{def-rho} is a $\kappa$-confinement of the random function $g$ on the manifold $V_k(\R^m)\times \R^k \times V_k(\R^n)$.
In particular, $\rho$ satisfies Inequalities \eqref{eq-def-confinement-1} and \eqref{eq-def-confinement-2} for any values of $\rho_0$ and $\rho_1$ satisfying Inequalities \eqref{eq-kappa-rho1-condition} and \eqref{eq-rho0-def}.
\end{proposition}

\begin{proof}
For any $r>0$, 
\[
\{(U,\mathbf{x},V)\in V_k(\R^m)\times \R^k \times V_k(\R^n)~\big{|}~\rho(U,\mathbf{x},V) \leq r^2\} = V_k(\R^m)\times \{\mathbf{x}\in \R^k ~\big{|}~\|\mathbf{x}\| \leq r\}\times V_k(\R^n),
\] 
which is compact. 

By Inequalities  \eqref{eq-kappa-rho1-condition} and \eqref{eq-kappa-lambda-constraint}, we have $\rho_1 \ge \rho_0 + a\kappa + (16ka + 8\lambda^2\rho_0 + 16k\rho_0)\kappa^2  > \rho_0$. By Lemma \ref{lemma-rho1-condition}, Inequality \eqref{eq-rho-R-le-rho1} follows from Inequality \eqref{eq-kappa-rho1-condition}. Also, by Inequalities \eqref{eq-kappa-lambda-constraint} and  \eqref{eq-rho0-def}, $\rho_0 \ge \frac{(1+16k\kappa)a}{4\lambda - 16k\kappa - 8\lambda^2 \kappa} \ge \frac{a}{4\lambda} >0$.

Inequality \eqref{eq-rho0-def} is equivalent to 
\begin{equation*}
\kappa \le \frac{4\lambda\rho_0-a}{(16k+8\lambda^2)\rho_0 + 16ka}.
\end{equation*}
Notice that $\frac{4\lambda\|\mathbf{x}\|^2-a}{(16k+8\lambda^2)\|\mathbf{x}\|^2 + 16ka}$ is an increasing function of $\|\mathbf{x}\|^2$ when $\|\mathbf{x}\|^2 \ge \frac{a}{4\lambda}$.
Then 
\begin{equation} \label{eq-kappa-value}
\kappa \le \frac{4\lambda\|\mathbf{x}\|^2-a}{(16k+8\lambda^2)\|\mathbf{x}\|^2 + 16ka}
\end{equation}
for $\rho_0 \le \|\mathbf{x}\|^2 \le \rho_1$.
Moreover, Inequality \eqref{eq-kappa-value} is equivalent to Inequality \eqref{eq-kappa-rho0-condition}.
Thus, by Lemma \ref{lemma-rho0-condition}, Inequality \eqref{eq-kappa-conf-hess-con} follows.
Therefore, $\rho(U,\mathbf{x},V)$ is a $\kappa$-confinement of $g$ according to Definition \ref{def-kappa-confinement}.
\end{proof}

\subsection{The Algorithm}

Now we fix the scalars associated to the $\kappa$-confinement in Assumption \eqref{assumption-confinement-function} of Theorem \ref{thm-Ada-SGD-mfd}.
For $\lambda$ and $k$ given in Problem \ref{prob-RWLRA-1-reform} and $a$ defined in Definition \ref{def-amax}, we fix a $\kappa$ satisfying Inequality \eqref{eq-kappa-lambda-constraint}, and set
\begin{equation} \label{eq-rho0-val}
\rho_0   =  \frac{(1+16k\kappa)a}{4\lambda - 16k\kappa - 8\lambda^2 \kappa}
\end{equation}
which satisfies Inequality \eqref{eq-rho0-def}. 
For $\mathbf{x}_0$ in the initial iterate $(U_0,\mathbf{x}_0,V_0)  \in V_k(\R^m)\times \R^k \times V_k(\R^n)$, we set
\begin{equation} \label{eq-rho1-val}
\rho_1   =  \max \{\| \mathbf{x}_0 \|^2  , \  \rho_0 + a\kappa + (16ka + 8\lambda^2\rho_0 + 16k\rho_0)\kappa^2  \}.
\end{equation}
Thus, Assumption (\ref{assumption-initial value}) in Theorem \ref{thm-Ada-SGD-mfd} and Inequality \eqref{eq-kappa-rho1-condition} are satisfied.
By Proposition \ref{prop-rho0-rho1-value}, $\rho(U,\mathbf{x},V) = \| \mathbf{x} \|^2$ defined in Definition \ref{def-rho} is a $\kappa$-confinement of the random function $g$ for $\kappa$, $\rho_0$ and $\rho_1$ satisfying the conditions in that proposition.

We fix the initial learning rate to be $\eta_0 = \frac{\alpha}{\beta^{\frac{1}{2}+\ve}} = \kappa$, where $\alpha>0$, $0 < \ve \le \frac{1}{2}$. This means that
\begin{equation} \label{eq-beta-val}
\beta =  \left( \frac{\alpha}{\kappa} \right)^{\frac{2}{1+2\ve}}.
\end{equation} 

With the preparations above, we are now ready to present Algorithm \ref{alg-Ada-SGD-RWLRA} for Problem \ref{prob-RWLRA-1-reform}.

\begin{algorithm}[Stochastic Gradient Descent with Adaptive Learning Rates for Problem \ref{prob-RWLRA-1-reform}]\label{alg-Ada-SGD-RWLRA}
Let $g$ be the random function defined in Definition \ref{def-random-functions-tau-gamma-RWLRA-1} and $R$ the retraction defined in Definition \ref{def-retraction-GS-prod}.

\noindent\makebox[\linewidth]{\rule{\textwidth}{1pt}}

\textbf{Input:}   
\begin{itemize}
\item[-] (From Problem \ref{prob-RWLRA-1-reform}) Positive integers $m,~n$ and $k$ satisfying $k \leq \min\{m,n\}$, a scalar $\lambda>0$, matrices of constants $A=[a_{i,j}] \in \R^{m\times n}$ and $W=[w_{i,j}] \in \R^{m\times n}$ such that $\sum_{i=1}^m\sum_{j=1}^n w_{i,j}=1$ and $w_{i,j} \geq 0$ $\forall ~i,j$. 
\item[-] ($\kappa$-Confinement) A scalar $\kappa$ satisfying $0 < \kappa < \frac{\lambda}{4k + 2\lambda^2}$ (Inequality \eqref{eq-kappa-lambda-constraint}) and $\rho_0 > 0$ given by Equation \eqref{eq-rho0-val}.\footnote{$\rho_1$ does not explicitly occur in Algorithm \ref{alg-Ada-SGD-RWLRA}. But with $\rho_1$ given by Equation \eqref{eq-rho1-val}, all the relevant assumptions in Theorem \ref{thm-Ada-SGD-mfd} are satisfied}
\item[-] (Initial Learning Rate) Positive scalars $\alpha>0$, $0 < \ve \le \frac{1}{2}$, and $\beta$ given by Equation \eqref{eq-beta-val}.
\item[-] (Initial Iterate) $(U_0,\mathbf{x}_0,V_0)  \in V_k(\R^m)\times \R^k \times V_k(\R^n)$.

\end{itemize}

\textbf{Output:} Sequence of iterates $\{(U_t,\mathbf{x}_t,V_t)\}_{t=0}^\infty  \subset V_k(\R^m)\times \R^k \times V_k(\R^n)$.
\begin{itemize}
	\item \emph{for $t=1,2\dots$ do}
	\begin{enumerate}[1.]
	\item Select a random element $(\tau_t, \gamma_t)$ from $\Omega_{m, n}$ with probability distribution $\mu$ independent of $\{(\tau_{\iota}, \gamma_{\iota})\}_{\iota=0}^{t-1}$.
	\item Update learning rate as
	\begin{equation}
	\eta_t = \frac{\alpha}{(\beta+\sum_{\iota=0}^{t-1} \|\nabla g_{\tau_\iota,\gamma_\iota}(U_\iota,\mathbf{x}_\iota,V_\iota)\|^2)^{\frac{1}{2}+\ve}},
	\end{equation}
	where $\nabla g_{\tau,\gamma}(U,\mathbf{x},V)$ is computed in Corollary \ref{cor-random-g-gradient-RWLRA-1}.
	\item Set 
	\begin{equation}\label{eq-ada-RWLRA-1-recursion}
	(U_{t+1},\mathbf{x}_{t+1},V_{t+1})= R_{(U_t,\mathbf{x}_t,V_t)}( -\eta_t \nabla g_{\tau_t,\gamma_t}(U_t,\mathbf{x}_t,V_t)).
	\end{equation}
	\end{enumerate}
	\item \emph{end for}
\end{itemize}
\noindent\makebox[\linewidth]{\rule{\textwidth}{1pt}}
\end{algorithm}

\begin{proposition}\label{prop-Ada-SGD-RWLRA-convergent}
Applying Theorem \ref{thm-Ada-SGD-mfd} to Algorithm \ref{alg-Ada-SGD-RWLRA}, we have
\begin{enumerate}
	\item $\{(U_t,\mathbf{x}_t,V_t)\}_{t=0}^\infty$ is contained in the compact subset $K_1= V_k(\R^m)\times \hat{K}_1 \times V_k(\R^n)$ of $V_k(\R^m)\times \R^k \times V_k(\R^n)$, where $\hat{K}_1 =\{ \mathbf{x} \in \R^k : \| \mathbf{x} \| ^2 \le \rho_1 \}$. Therefore, $\{(U_t,\mathbf{x}_t,V_t)\}_{t=0}^\infty$ has convergent subsequences.
	\item $\{\|\nabla G(U_t,\mathbf{x}_t,V_t)\|\}_{t=0}^\infty$ converges almost surely to $0$.
	\item Any limit point of $\{(U_t,\mathbf{x}_t,V_t)\}_{t=0}^\infty$ is almost surely a stationary point of $G$.
\end{enumerate}
\end{proposition}

\begin{remark} \label{rmk-lmbd}
Observe that, by Inequality \eqref{eq-kappa-lambda-constraint}, we have that $\eta_t \le \eta_0 \le \kappa = O(\lambda)$ as $\lambda \to 0$. If the regularizing scalar $\lambda$ is too small, then the learning rates would be very small and slow down the convergence. This is observed in the numerical results from Section \ref{sec-numerical}.
\end{remark}

\section{Numerical Results}\label{sec-numerical}

Netflix provided a dataset \cite{kaggle-Netflix-Prize:2024} with $100480507$ integer ratings (ranging from $1$ to $5$) given by $480189$ users on $17770$ movies, forming a $480189\times 17770$ data matrix with about $1.17\%$ of its entries observed. 
In this section, we evaluate the performance of Algorithm \ref{alg-Ada-SGD-RWLRA} on the Matrix Completion Problem with a subset of the Netflix data and compare it to that of \cite[Algorithm 3.13]{Xu-Yang-Wu-CSGD}, which is a convergent stochastic gradient descent algorithm on $V_k(\R^m)\times \R^k \times V_k(\R^n)$ with deterministic learning rates. 
It is worth mentioning that \cite[Algorithm 3.13]{Xu-Yang-Wu-CSGD} out performs a convergent stochastic gradient descent algorithm on the Euclidean space $\R^{m\times k} \times \R^{n\times k}$. See \cite[Problem 4.1, Figure 1 and Algorithm D.13]{Xu-Yang-Wu-CSGD}.

Due to our limited computing power, instead of working with the whole data matrix, we work with a randomly sampled submatrix $A = [a_{i,j}] \in \R^{27000\times 1000}$ of the whole data matrix. $A$ has $278338$ observed entries, which is about $1.03\%$ of all of its entries. We fix the low-rank constraint to be $k = 32$. We also fix the matrix of weights to be 
$$W = [w_{i,j}] \in \R^{27000\times 1000}, \text{ where} \ w_{i,j} = \begin{cases} 0 & \text{if } a_{i,j} \text{ is missing,} \\ \frac{1}{278338} &  \text{if } a_{i,j} \text{ is observed.} \end{cases}$$
These choices for $A$, $k$ and $W$ are used for both Algorithm \ref{alg-Ada-SGD-RWLRA} and \cite[Algorithm 3.13]{Xu-Yang-Wu-CSGD}. Denote by $\Omega_{27000,1000}$ the probability space defined in Definition \ref{def-mu-prob-space} with its probability measure given by the weight matrix $W = [w_{i,j}]$. We then randomly sample a sequence $\{(i_t,j_t)\}_{t=0}^{1000}$ of elements of $\Omega_{27000,1000}$ to use for the iterations of both Algorithm \ref{alg-Ada-SGD-RWLRA} and \cite[Algorithm 3.13]{Xu-Yang-Wu-CSGD}.

In Figures \ref{fig-ASGD-vs-CSGD-netflix} and \ref{fig-ASGD-vs-CSGD-netflix-sp}, we refer to 
\begin{itemize}
	\item $\hat{F}$ from Problem \ref{prob-WLRA} as ``the cost function with regularization removed'',
	\item Algorithm \ref{alg-Ada-SGD-RWLRA} as ``the Adaptive SGD'',
	\item\cite[Algorithm 3.13]{Xu-Yang-Wu-CSGD} as ``the Deterministic SGD''.
\end{itemize}
In Figure \ref{fig-ASGD-vs-CSGD-netflix}, We plot $\{ \hat{F}(U_tD_{32}^{32 \times 32}(\mathbf{x}_t)V_t^T) \}_{t=0}^{1000}$ from Problem \ref{prob-WLRA} for both the Adaptive and the Deterministic SGD algorithms for $\lambda = 10^{-4}$, $10^{-6}$ and $10^{-8}$.
\begin{itemize}
	\item For the Adaptive SGD, we fix $\alpha = 10^5$ and $\varepsilon = \frac{1}{64}$.\footnote{It is worth mentioning that the choice of $\alpha$ in Algorithm \ref{alg-Ada-SGD-RWLRA} barely affects the performance of the algorithm when it is chosen properly as a positive number.} Corresponding to the three values of $\lambda$, the confinement scalar $\kappa$ is chosen to be $10^{-7}$, $5 \times 10^{-9}$ and $5 \times 10^{-11}$ to optimize the performance of the algorithm. Note that these choices of $\kappa$ satisfy the condition $0 < \kappa < \frac{\lambda}{128 + 2\lambda^2}$ imposed by Inequality \eqref{eq-kappa-lambda-constraint}.
	\item For the Deterministic SGD, the scalar $K$ is fixed as $10^4$ for all three values of $\lambda$ to optimize the performance of the algorithm.
\end{itemize}

\begin{figure}[ht]
\centering
\includegraphics[width=7cm, height=7cm]{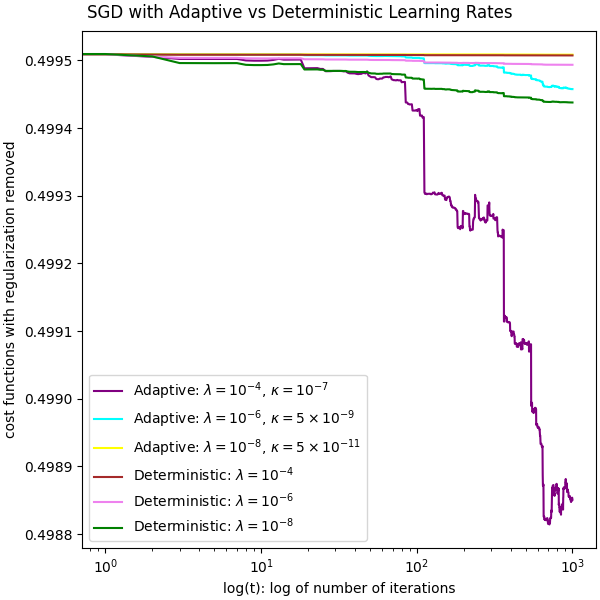}
\caption{The performance profile of the Adaptive and the Deterministic SGD.}
\label{fig-ASGD-vs-CSGD-netflix}
\end{figure}

With these choices of parameters, curves in Figure \ref{fig-ASGD-vs-CSGD-netflix} depict the sequence $\{ \hat{F}(U_tD_{32}^{32 \times 32}(\mathbf{x}_t)V_t^T) \}_{t=0}^{1000}$ over the $\log_{10}$ of iteration $t$. One can see that the Adaptive SGD with $\lambda = 10^{-4}$ and $\kappa = 10^{-7}$ (the purple curve) has the best performance among the six implementations of the two algorithms. Note that the yellow curve (Adaptive SGD: $\lambda = 10^{-8}$ and $\kappa = 5 \times 10^{-11}$) is extremely close to the brown curve (Deterministic SGD: $\lambda = 10^{-4}$). For these two choices of parameters, the two algorithms barely change the value $\hat{F}$. This seems to confirm Remark \ref{rmk-lmbd}.

\begin{figure}[ht]
\centering
\includegraphics[width=15cm, height=5cm]{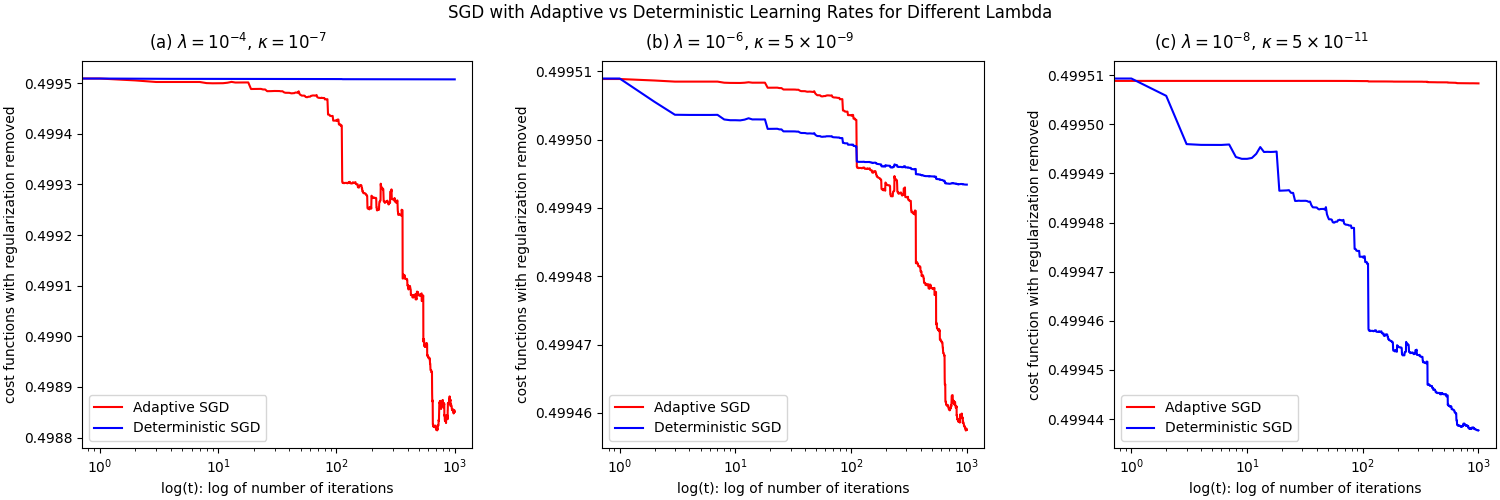}
\caption{The performance profiles of the Adaptive and the Deterministic SGD for different values of $\lambda$.}
\label{fig-ASGD-vs-CSGD-netflix-sp}
\end{figure}

Figure \ref{fig-ASGD-vs-CSGD-netflix-sp} provides a closer look at the same curves in Figure \ref{fig-ASGD-vs-CSGD-netflix} sorted by different $\lambda$ values. 
Once can see that the Adaptive SGD outperforms the Deterministic SGD for $\lambda = 10^{-4}$ and $10^{-6}$, but underperforms the Deterministic SGD for $\lambda = 10^{-8}$. 
This seems to indicate that, as $\lambda \rightarrow 0$, the Adaptive SGD gradually loses its efficientcy relative to the Deterministic SGD.
The Python implementation of the algorithms in this numerical section is available on Github: \url{https://github.com/pqyang9/Adaptive-SGD}.
 
\newpage

\bibliography{rie-opt-refs}    

\newcommand{\etalchar}[1]{$^{#1}$}
\begin{thebibliography}{BCWLP24}

\bibitem[AIS98]{Alber-Iusem-Solodov:1998}
Ya.~I. Alber, A.~N. Iusem, and M.~V. Solodov.
\newblock On the projected subgradient method for nonsmooth convex optimization
  in a {H}ilbert space.
\newblock {\em Math. Programming}, 81(1):23--35, 1998.

\bibitem[AMS09]{AMS}
P.A. Absil, R.~Mahony, and R.~Sepulchre.
\newblock {\em Optimization Algorithms on Matrix Manifolds}.
\newblock Princeton University Press, 2009.

\bibitem[BA11]{Boumal-Absil:2011}
Nicolas Boumal and P.A. Absil.
\newblock Rtrmc: A riemannian trust-region method for low-rank matrix
  completion.
\newblock In J.~Shawe-Taylor, R.~Zemel, P.~Bartlett, F.~Pereira, and K.Q.
  Weinberger, editors, {\em Advances in Neural Information Processing Systems},
  volume~24. Curran Associates, Inc., 2011.

\bibitem[BCN23]{Bordenave-Coste-Nadakuditi:2023}
Charles Bordenave, Simon Coste, and Raj~Rao Nadakuditi.
\newblock Detection thresholds in very sparse matrix completion.
\newblock {\em Found. Comput. Math.}, 23(5):1619--1743, 2023.

\bibitem[BCWLP24]{Bertsimas-Cory-Lo-Pauphilet:2024}
Dimitris Bertsimas, Ryan Cory-Wright, Sean Lo, and Jean Pauphilet.
\newblock Optimal low-rank matrix completion: Semidefinite relaxations and
  eigenvector disjunctions, 2024.

\bibitem[BT00]{Bertsekas-Tsitsiklis:2000}
Dimitri~P. Bertsekas and John~N. Tsitsiklis.
\newblock Gradient convergence in gradient methods with errors.
\newblock {\em SIAM J. Optim.}, 10(3):627--642, 2000.

\bibitem[BWZ19]{Ban-Woodruff-Zhang:2019}
Frank Ban, David Woodruff, and Richard Zhang.
\newblock Regularized weighted low rank approximation.
\newblock In H.~Wallach, H.~Larochelle, A.~Beygelzimer, F.~d\textquotesingle
  Alch\'{e}-Buc, E.~Fox, and R.~Garnett, editors, {\em Advances in Neural
  Information Processing Systems}, volume~32. Curran Associates, Inc., 2019.

\bibitem[CD23]{Chakraborty-Dey:2023}
Diptarka Chakraborty and Sanjana Dey.
\newblock Matrix completion: approximating the minimum diameter.
\newblock In {\em 34th {I}nternational {S}ymposium on {A}lgorithms and
  {C}omputation}, volume 283 of {\em LIPIcs. Leibniz Int. Proc. Inform.}, pages
  Art. No. 17, 19. Schloss Dagstuhl. Leibniz-Zent. Inform., Wadern, 2023.

\bibitem[GG11]{Gillis-Glineur:2011}
Nicolas Gillis and Fran\c{c}ois Glineur.
\newblock Low-rank matrix approximation with weights or missing data is
  np-hard.
\newblock {\em SIAM Journal on Matrix Analysis and Applications},
  32(4):1149--1165, 2011.

\bibitem[Kag24]{kaggle-Netflix-Prize:2024}
Kaggle.
\newblock {N}etflix {P}rize data --- kaggle.com, 2024.
\newblock [accessed 01-12-2024].

\bibitem[KLL{\etalchar{+}}23]{Kelner-Li-Liu-Sidford-Tian:2023}
Jonathan~A. Kelner, Jerry Li, Allen Liu, Aaron Sidford, and Kevin Tian.
\newblock Matrix completion in almost-verification time.
\newblock In {\em 2023 {IEEE} 64th {A}nnual {S}ymposium on {F}oundations of
  {C}omputer {S}cience---{FOCS} 2023}, pages 2102--2128. IEEE Computer Soc.,
  Los Alamitos, CA, 2023.

\bibitem[Lee24]{Lee:2024}
Geunseop Lee.
\newblock Smooth singular value thresholding algorithm for low-rank matrix
  completion problem.
\newblock {\em J. Korean Math. Soc.}, 61(3):427--444, 2024.

\bibitem[LO19]{Li-Orabona:2019}
Xiaoyu Li and Francesco Orabona.
\newblock On the convergence of stochastic gradient descent with adaptive
  stepsizes.
\newblock In Kamalika Chaudhuri and Masashi Sugiyama, editors, {\em Proceedings
  of the Twenty-Second International Conference on Artificial Intelligence and
  Statistics}, volume~89 of {\em Proceedings of Machine Learning Research},
  pages 983--992. PMLR, 2019.

\bibitem[Mai13]{Mairal:2013}
Julien Mairal.
\newblock Stochastic majorization-minimization algorithms for large-scale
  optimization.
\newblock In C.J. Burges, L.~Bottou, M.~Welling, Z.~Ghahramani, and K.Q.
  Weinberger, editors, {\em Advances in Neural Information Processing Systems},
  volume~26. Curran Associates, Inc., 2013.

\bibitem[SJ03]{Srebro-Jaakkola:2003}
Nathan Srebro and Tommi Jaakkola.
\newblock Weighted low-rank approximations.
\newblock In {\em Proceedings of the Twentieth International Conference on
  International Conference on Machine Learning}, ICML'03, page 720–727. AAAI
  Press, 2003.

\bibitem[Tay12]{Taylor}
Michael~E. Taylor.
\newblock {\em Partial Differential Equations I}, volume 115 of {\em App. Math.
  Sci.}
\newblock Springer New York, NY, 2012.

\bibitem[WC23]{Wan-Cheng:2023}
Xiying Wan and Guanghui Cheng.
\newblock Weighted hybrid truncated norm regularization method for low-rank
  matrix completion.
\newblock {\em Numer. Algorithms}, 94(2):619--641, 2023.

\bibitem[XYW25]{Xu-Yang-Wu-CSGD}
Conglong Xu, Peiqi Yang, and Hao Wu.
\newblock Weighted low-rank approximation via stochastic gradient descent on
  manifolds, 2025.
\newblock arXiv:2502.14174.

\bibitem[YM23]{Yang-Ma:2023}
Yuepeng Yang and Cong Ma.
\newblock Optimal tuning-free convex relaxation for noisy matrix completion.
\newblock {\em IEEE Trans. Inform. Theory}, 69(10):6571--6585, 2023.

\bibitem[YTL24]{Yan-Tang-Li:2024}
Xihong Yan, Xiaoni Tang, and Chao Li.
\newblock An improved inertial alternating direction method for low rank matrix
  completion problems.
\newblock {\em Math. Numer. Sin.}, 46(2):144--155, 2024.

\bibitem[YZ24]{Yan-Zhang:2024}
Xihong Yan and Ning Zhang.
\newblock A modified primal-dual algorithm for matrix completion problems.
\newblock {\em Acta Math. Appl. Sin.}, 47(2):175--192, 2024.

\end{thebibliography}

\end{document}